\documentclass[11pt]{article}
\usepackage{amssymb}
\usepackage{amsmath}
\usepackage{amsthm}
\usepackage{verbatim}
\usepackage{graphicx}
\usepackage{tikz,adjustbox}
\usepackage{mathtools}
\usepackage{epstopdf}
\usepackage{color}
\usepackage{epsf}
\usepackage[margin=1.183in]{geometry}
\usepackage{subfig}

\usepackage{enumerate}
\usepackage{enumitem}
\setlist{itemsep=0pt, topsep=0pt}

\usepackage{hyperref}

\newtheorem{theorem}{Theorem}

\newtheorem{thm}{}[section]
\newtheorem{lemma}[thm]{Lemma}

\newtheorem{conjecture}[theorem]{Conjecture}
\newtheorem{corollary}[thm]{Corollary}
\newtheorem{observation}[thm]{Observation}
\newtheorem{proposition}[thm]{Proposition}

\numberwithin{subcase}{case}

\theoremstyle{definition}
\newtheorem{definition}[thm]{Definition}

\DeclareMathOperator{\ESO}{ESO}
\DeclareMathOperator{\ES}{ES}
\renewcommand{\P}{\mathcal{P}}
\newcommand{\C}{\mathcal{C}}
\newcommand{\B}{\mathcal{B}}
\newcommand{\A}{\mathcal{A}}
\newcommand{\U}{\mathcal{U}}
\newcommand{\V}{\mathcal{V}}
\newcommand{\N}{\mathcal{N}}
\newcommand{\W}{\mathcal{W}}
\renewcommand{\S}{\mathcal{S}}
\renewcommand{\L}{\mathcal{L}}
\newcommand{\R}{\mathcal{R}}
\DeclareMathOperator{\NE}{NE}
\DeclareMathOperator{\NW}{NW}
\DeclareMathOperator{\SE}{SE}
\DeclareMathOperator{\SW}{SW}
\DeclareMathOperator*{\argmax}{argmax}
\DeclareMathOperator*{\argmin}{argmin}
\renewcommand{\t}{\texttt{t}}
\newcommand{\cupdot}{\mathbin{\mathaccent\cdot\cup}}

\begin{document}

\title{Erd\H{o}s-Szekeres On-Line}

\author{Kirk Boyer\thanks{Department of Mathematics, University of Denver, Denver, CO 80208; {\tt kirk.boyer@du.edu, lauren.nelsen@du.edu}.} \and Lauren M. Nelsen\footnotemark[1] \and Luke L. Nelsen\thanks{Department of Mathematical and Statistical Sciences, University of Colorado Denver, Denver, CO 80217; {\tt luke.nelsen@ucdenver.edu, florian.pfender@ucdenver.edu}.} \and Florian Pfender\footnotemark[2] \and Elizabeth Reiland\thanks{Department of Applied Mathematics and Statistics, Johns Hopkins University, Baltimore, MD 21218; {\tt ereiland@jhu.edu}.} \and Ryan Solava\thanks{Department of Mathematics, Vanderbilt University, Nashville, TN 37240; {\tt ryan.w.solava@vanderbilt.edu}.}}

\maketitle

\begin{abstract}
In 1935, Erd\H{o}s and Szekeres proved that $(m-1)(k-1)+1$ is the minimum number of points in the plane which definitely contain an increasing subset of $m$ points or a decreasing subset of $k$ points (as ordered by their $x$-coordinates).
We consider their result from an on-line game perspective:
Let points be determined one by one by player A first determining the $x$-coordinate and then player B determining the $y$-coordinate.
What is the minimum number of points such that player A can force an increasing subset of $m$ points or a decreasing subset of $k$ points?
We introduce this as the \textit{Erd\H{o}s-Szekeres on-line number} and denote it by $\text{ESO}(m,k)$.
We observe that $\text{ESO}(m,k) < (m-1)(k-1)+1$ for $m,k \ge 3$, provide a general lower bound for $\text{ESO}(m,k)$, and determine $\text{ESO}(m,3)$ up to an additive constant.
\end{abstract}

\section{Introduction}

In \cite{ErdosSzekeres1935}, Erd\H{o}s and Szekeres proved that $(m-1)(k-1)+1$ is the minimum number of points in the plane (ordered by their $x$-coordinates) that guarantees an increasing (in terms of $y$-coordinates) subset of $m$ points or a decreasing subset of $k$ points.
We refer to this number as the \textit{Erd\H{o}s-Szekeres number} and denote it by $\ES(m,k)$.
Their theorem has since seen several proofs, as well as related results with random or algorithmic themes (see \cite{Steele1995}).

``On-line'' refers to a process in which an entire structure is not known, and instead decisions must be made with limited information.
On-line graph coloring was most notably developed in \cite{GyarfasLehel1988} by Gyarf\'as and Lehel; for more developments which have sprung from this topic, we refer the reader to \cite{Kierstead1998,Paulusma2016}.
We consider the question of Erd\H{o}s and Szekeres in an on-line setting with the following game:
Let points be determined one by one with player A first determining the $x$-coordinate and then player B determining the $y$-coordinate.
The question we want to answer is the following.
What is the minimum number of points such that player A can force an increasing subset of $m$ points or a decreasing subset of $k$ points?
We refer to this number as the \textit{Erd\H{o}s-Szekeres on-line number} and denote it by $\ESO(m,k)$.

In Section~\ref{sec:prelim} we introduce necessary definitions, a table of small results, and prove the following weak but general upper bound:

\begin{theorem}\label{thm:ESO<ES}
$\ESO(m,k) \le (m-1)(k-1)$ for all $m,k \ge 3$.
\end{theorem}

If both players are playing uniformly at random, it is not difficult to see that the game for $m=k$ typically ends after about $\frac12m^2$ turns. Considering random play often gives good intuition for bounds in deterministic play, and this random intuition would suggest that the bound in Theorem~\ref{thm:ESO<ES} is off by a factor $2$.

We establish such a lower bound for the Erd\H{o}s-Szekeres on-line number in Section~\ref{sec:mkLB}:

\begin{theorem}\label{thm:mkLB}
$\ESO(m,k) \ge \lfloor\frac{k}{2}\rfloor (m-k+5) -3$ for $m \ge k \ge 4$.
\end{theorem}

Notice the unusual behavior of this bound depending on the parity of $k$.
We achieve this lower bound by considering a related game which restricts the choices for player B. For this related game, we show that the leading term in Theorem~\ref{thm:mkLB} is correct for every fixed $k$, and the dependence on the parity of $k$ can be seen in the proof. In fact, we conjecture that this leading term is correct for the original game as well.

\begin{conjecture}\label{con}
For all $m\ge k$, we have $\ESO(m,k)=\lfloor\frac{k}{2}\rfloor m + O(k^2)+o(mk)$.
\end{conjecture}

For the values $k=1$ and $k=2$, it is trivial to determine $\ESO(m,k)$. The first interesting value is $k=3$. We find evidence for our conjecture by providing strategies for player A (Section~\ref{sec:m3UB}) and player B (Section~\ref{sec:m3LB}), obtaining the following result.

\begin{theorem}\label{thm:m3} 
$\ESO(m,3) = m + (6m)^{\frac13} + O(1).$
Specifically,
\[
m + (6m)^{\frac13} - 2 < \ESO(m,3) < m + (6m)^{\frac13} + 3.
\]
\end{theorem}

Finally, we mention some variations of this on-line game in Section~\ref{sec:conclusion}.

\section{Preliminary Definitions and Observations}\label{sec:prelim}

We begin with a formal definition of the \textit{Erd\H{o}s-Szekeres on-line number}.

\begin{definition}\textit{$x(p)$, $y(p)$, Up-runs and down-runs.}\\
Given a point $p \in \mathbb{R}^2$, we denote the $x$-coordinate of $p$ by $x(p)$ and the $y$-coordinate by $y(p)$.
When we denote a set of points by $p_1, \dots, p_n$, we assume that $x(p_i)<x(p_{i+1})$ for all $1\le i\le n-1$.
Consider a set of points $\C = \{p_1,\ldots,p_n\}$.
If $y(p_i)\le y(p_{i+1})$ for all $1\le i\le n-1$, then we say that $\C$ is an \textit{$n$-up-run}, or simply an \textit{up-run}.
If $y(p_i)> y(p_{i+1})$ for all $1\le i\le n-1$, then we say that $\C$ is an \textit{$n$-down-run}, or simply a \textit{down-run}.
\end{definition}

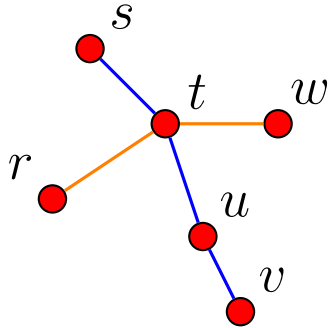
\begin{figure}[h!]
\center
\begin{adjustbox}{max totalsize={6cm}{6cm},center,padding=1em}
\begin{tikzpicture}

\begin{scope}[every node/.style={circle,thick,draw,shade,shading=axis,left color=red,right color=red,shading angle=90}]
    \node (1) at (0,1.5) {};
    \node (2) at (.5,3.5) {};
    \node (3) at (1.5,2.5) {};
    \node (4) at (2,1) {};
    \node (5) at (2.5,0) {};
    \node (6) at (3,2.5) {};
\end{scope}

\node at ([shift={(135:.6)}]1) {\huge $r$};
\node at ([shift={(45:.6)}]2) {\huge $s$};
\node at ([shift={(45:.6)}]3) {\huge $t$};
\node at ([shift={(45:.6)}]4) {\huge $u$};
\node at ([shift={(45:.6)}]5) {\huge $v$};
\node at ([shift={(45:.6)}]6) {\huge $w$};

\begin{scope}[-,every edge/.style={draw=orange,very thick}]
    \path [-] (1) edge node {} (3);
    \path [-] (3) edge node {} (6);
\end{scope}

\begin{scope}[-,every edge/.style={draw=blue,very thick}]
    \path [-] (2) edge node {} (3);
    \path [-] (3) edge node {} (4);
    \path [-] (4) edge node {} (5);
\end{scope}

\end{tikzpicture}
\end{adjustbox}
\caption{The set $\{r,t,w\}$ is a 3-up-run, and the set $\{s,t,u,v\}$ is a 4-down-run.}
\end{figure}

\begin{definition}\textit{The Erd\H{o}s-Szekeres on-line game, $A_{m,k}$.}\\
Let $m,k \ge 1$.
In each step of the game $A_{m,k}$, player A chooses a value $\hat{x} \in (0,1)$ and then player B chooses a value $\hat{y} \in (0,1)$, forming a point $(\hat{x},\hat{y})$ in $(0,1) \times (0,1)$.
The game ends when after some turn there is either an $m$-up-run or a $k$-down-run.
Let player A have the objective of ending the game in the fewest number of turns, and player B in the greatest number of turns.
We denote by $\ESO(m,k)$ the number of turns a game of $A_{m,k}$ will take when both players play optimally, which we call the \textit{Erd\H{o}s-Szekeres on-line number}.
\end{definition}

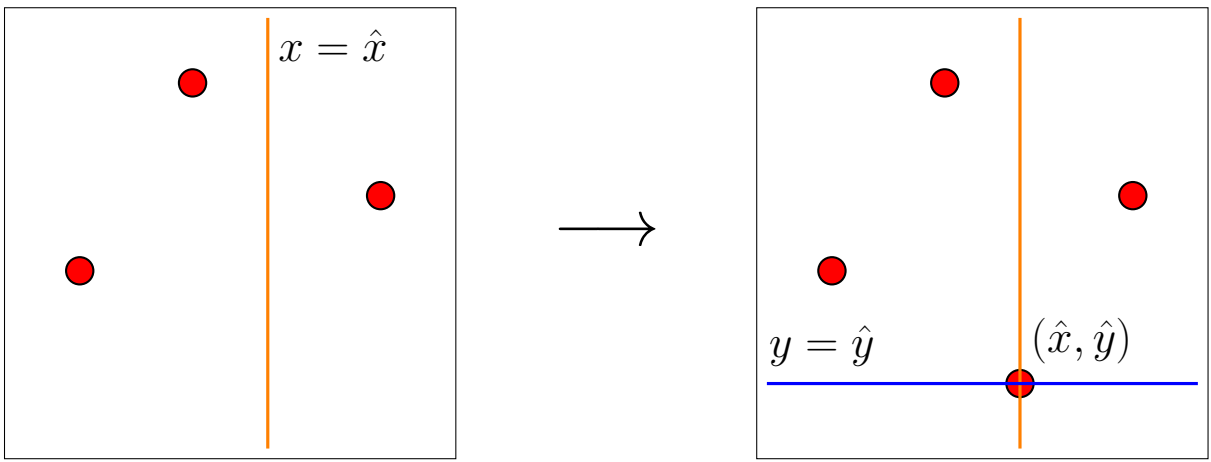
\begin{figure}[h!]
\center
\begin{tikzpicture}

	\draw (-3,-3) rectangle (3,3);
    \node (1) at (0.5,-3) {};
    \node (2) at (0.5,3) {};
	\node at ([shift={(-30:1)}]2) {\LARGE $x = \hat{x}$};
    
    \node at (5,0) {\Huge $\longrightarrow$};

	\draw (7,-3) rectangle (13,3);
    \node (11) at (10.5,-3) {};
    \node (12) at (10.5,3) {};
    \node (13) at (7,-2) {};
    \node (14) at (13,-2) {};
	\node at ([shift={(30:1)}]13) {\LARGE $y = \hat{y}$};
    
\begin{scope}[every node/.style={circle,thick,draw,shade,shading=axis,left color=red,right color=red,shading angle=90}]
    \node (5) at (-.5,2) {};
    \node (6) at (-2,-.5) {};
    \node (7) at (2,.5) {};
    
    \node (15) at (9.5,2) {};
    \node (16) at (8,-.5) {};
    \node (17) at (12,.5) {};
    \node (18) at (10.5,-2) {};
\end{scope}

	\node at ([shift={(35:1)}]18) {\LARGE $(\hat{x},\hat{y})$};

\begin{scope}[-,
              every edge/.style={draw=orange,very thick}]
    \path [-] (1) edge node {} (2);
    \path [-] (11) edge node {} (12);
\end{scope}

\begin{scope}[-,
              every edge/.style={draw=blue,very thick}]
    \path [-] (13) edge node {} (14);
\end{scope}
\end{tikzpicture}

\caption{First, player A chooses $\hat{x}$, and then player B chooses $\hat{y}$.
Together this constitutes a turn in which the point $(\hat{x},\hat{y})$ was played.}
\end{figure}

Clearly, $\ESO(m,k) \ge \min\{m,k\}$.
Note that in an optimal strategy, neither player needs to repeat a previously played $x$- or $y$-value; when providing a strategy for one player, we will assume that the opposing player does not repeat a value already chosen.
By reflecting any instance of $A_{m,k}$ about the line $x = \frac{1}{2}$, we also see that $\ESO(m,k) = \ESO(k,m)$.
Hence we assume that $m \ge k$ unless otherwise noted.

Observe that, after the first turn, player B always has a choice which will not increase the length of a longest up-run.
To see this, without loss of generality assume that there is a point $p$ immediately to the right of the $x$-value player A has chosen.
Then player B can choose $y(p)+\varepsilon$ for a sufficiently small, positive $\varepsilon$ such that $p$ and the new point are interchangeable in any up-run containing one of them.
Similarly, player B always has a choice which will not increase the length of a longest down-run.
Thus $\ESO(m,k)$ is strictly increasing in both $m$ and $k$ when $m,k \ge 2$.

We also have $\ESO(m,k) \le \ES(m,k)$ and thus by the theorem of Erd\H{o}s and Szekeres \cite{ErdosSzekeres1935} we know that $\ESO(m,k) \le (m-1)(k-1)+1$.
It is clear that $\ESO(m,k) = \ES(m,k)$ when $k \in \{1,2\}$, but on the other hand it is easy (and perhaps a fun exercise for the reader) to show that $\ESO(3,3) = 4 < \ES(3,3) = 5$.
In fact, equality holds only for $k \le 2$.
We prove this after some preliminary definitions:

\begin{definition}\textit{Quadrants of a point.}\\
Let $\P$ be a point set in $(0,1) \times (0,1)$ and let $p' \in \P$.
Then the \textit{north-east quadrant of $p'$} is ~$\NE(p')  = \{ p \in \P : y(p') < y(p) ~ \& ~ x(p') < x(p) \}$.
The \textit{north-west quadrant}, \textit{south-west quadrant} and \textit{south-east quadrant of $p'$} are defined similarly and are denoted by $\NW(p')$, $\SW(p')$ and $\SE(p')$, respectively.
\end{definition}

\begin{figure}[h!]
\center
\begin{tikzpicture}

	\draw (0,0) rectangle (7,7);
    \node (11) at (0,4) {};
    \node (12) at (7,4) {};
    \node (13) at (4,0) {};
    \node (14) at (4,7) {};
    
\begin{scope}[every node/.style={circle,thick,draw,shade,shading=axis,left color=red,right color=red,shading angle=90}]
    \node (1) at (1,1) {};
    \node (2) at (2,5) {};
    \node (3) at (3,6) {};
    \node (4) at (4,4) {};
    \node (5) at (5,3) {};
    \node (6) at (6,2) {};
\end{scope}

	\node at ([shift={(40:.6)}]1) {\LARGE $p_1$};
	\node at ([shift={(40:.6)}]2) {\LARGE $p_2$};
	\node at ([shift={(40:.6)}]3) {\LARGE $p_3$};
	\node at ([shift={(40:.6)}]4) {\LARGE $p_4$};
	\node at ([shift={(40:.6)}]5) {\LARGE $p_5$};
	\node at ([shift={(40:.6)}]6) {\LARGE $p_6$};
	
	\node at (0,0)[shift={(35:.45)}] {SW};
	\node at (7,0)[shift={(155:.45)}] {SE};
	\node at (0,7)[shift={(-35:.45)}] {NW};
	\node at (7,7)[shift={(-155:.45)}] {NE};

\begin{scope}[-,
              every edge/.style={draw=gray,very thick}]
    \path [-] (11) edge node {} (12);
    \path [-] (13) edge node {} (14);
\end{scope}

\end{tikzpicture}
\caption{$\NE(p_4) = \varnothing$, $\NW(p_4) = \{p_2,p_3\}$, $\SW(p_4) = \{p_1\}$, and $\SE(p_4) = \{p_5,p_6\}$.}
\end{figure}

With these definitions, we are now ready to prove Theorem~\ref{thm:ESO<ES}.
Our proof is an adaptation of Seidenberg's proof that $\ES(n,n) \le (n-1)^2 +1$ (see \cite{Steele1995}).

\begin{proof}
Assume $\P$ is a set of $(m-1)(k-1)-1$ points that have already been played.
If $\P$ contains an $m$-up-run or a $k$-down-run, then we are done; so suppose otherwise.

Label each point $p \in \P$ with $(i,j)$, where $i$ is the length of the longest up-run in $\P$ with $p$ as its left-most point and $j$ is the length of the longest down-run in $\P$ with $p$ as its left-most point.
Since there are no $m$-up-runs and no $k$-down-runs, the set of labels of $\P$ is a subset of $[m-1] \times [k-1]$.

Observe that each of the labels are distinct as for any points $p,q$ with $x(p)<x(q)$, either the first coordinate of $q$'s label is larger than the first coordinate of $p$'s label or the second coordinate of $q$'s label is larger than the second coordinate of $p$'s label.
Hence only one element from $[m-1] \times [k-1]$ is missing as a label.

\textbf{Case 1:}
The missing label is not $(m-1,k-1)$.
Then there is some point $q$ that is the left-most point of both an $(m-1)$-up-run and a $(k-1)$-down-run.
If player A plays to the left of $q$, then player B will choose a $y$-value that results in either an $m$-up-run or a $k$-down-run.

\textbf{Case 2:}
The missing label is $(m-1,k-1)$.
Then there is some point $q$ that is labeled with $(1,1)$.
Now let $\P^\uparrow = \{ p \in \P : \NE(p) = \varnothing \}$.
Observe that $\P^\uparrow$ is the set of points whose labels have ``1'' in the first coordinate.
Since $(m-1,k-1)$ is the only missing label, the set of labels of points in $\P^\uparrow$ is $\{ (1,j) \}_{j=1}^{k-1}$.
Hence $|\P^\uparrow|=k-1$.
Also observe that no two points from $\P^\uparrow$ form a 2-up-run and thus $\P^\uparrow$ is a $(k-1)$-down-run with $q$ as the right-most point.
Making similar observations about the set $\P_\downarrow = \{ p \in \P : \SE(p) = \varnothing \}$ shows that $\P_\downarrow$ is an $(m-1)$-up-run with $q$ as the right-most point.
If player A plays to the right of $q$, then player B will choose a $y$-value that results in either an $m$-up-run or a $k$-down-run.
\end{proof}

Although Theorem~\ref{thm:ESO<ES} establishes that this on-line number of the Erd\H{o}s-Szekeres problem is distinct from their classical result when $m,k \ge 3$, it does not provide a precise sense of what $\ESO(m,k)$ is in general.
In addition to the degenerate cases, we list some small results obtained via dynamic programming\footnote{The Python code used to obtain these results in Table~\ref{tab:smallresults} can be viewed online at \href{http://math.ucdenver.edu/~nelsenl/projects/ErdosSzekeresOnline}{math.ucdenver.edu/\~{}nelsenl/projects/ErdosSzekeresOnline}.} in Table~\ref{tab:smallresults}.

\begin{table}[h!]
\caption{Some Small Exact Results}
\centering
\begin{tabular}{|c|c|c|c|c|c|c|c|c|c|}
\hline $(m,k)$, $m\ge k$ & $(m,1)$ & $(m,2)$ & $(3,3)$ & $(4,3)$ & $(5,3)$ & $(6,3)$ & $(7,3)$ & $(4,4)$ & $(5,4)$ \\
\hline $\ESO(m,k)$ & 1 & $m$ & 4 & 6 & 7 & 9 & 10 & 8 & 11\\
\hline $\ES(m,k)$ & 1 & $m$ & 5 & 7 & 9 & 11 & 13 & 10 & 13\\
\hline
\end{tabular}\label{tab:smallresults}
\end{table}

\section{A General Lower Bound for $\ESO(m,k)$}\label{sec:mkLB}

In this section, we study a closely related game $B_{m,k}$ (for $k \ge 2$).
The game is precisely the same as $A_{m,k}$, except that player B is always restricted to choose from the set $\{1,2,\ldots,k-1\}$ (referred to as \textit{tiers}).
Note that a $k$-down-run is impossible in this game, and hence player A's objective is to force an $m$-up-run while player B's is to avoid one.
Call the minimum number of moves in which player A can force a win in this game $B(m,k)$.
Player B can use a strategy from this game to play the previous game $A_{m,k}$.
Therefore, $\ESO(m,k)\ge B(m,k)$ for all $m$ and $k$.
Hence Theorem~\ref{thm:mkLB} is corollary to Proposition~\ref{prop:BgameLB2}.

For the game $B_{m,k}$, we will use a few more terms.

\begin{definition}\textit{$\L_x$, $\R_x$, Separated}\\
At any point during the game, we denote by $\L_x$ the set of points $p$ to the left of $x$, i.e. $x(p) < x$.
Similarly, we denote by $\R_x$ the set of points to the right of $x$.
Two $x$-values $x_1$ and $x_2$ are \textit{separated} by $t$ points if $|\R_{x_1} \cap \L_{x_2}| + |\R_{x_2} \cap \L_{x_1}| = t$.
To say that points $p_1$ and $p_2$ are \textit{separated} is to say that $x(p_1)$ and $x(p_2)$ are separated.
\end{definition}

\begin{proposition}\label{prop:BgameUB}
\[
B(m,k) \le \left\lfloor\frac{k}{2}\right\rfloor (m-1)+1.
\]
\end{proposition}

\begin{proof}
We describe a strategy for player A.
Player A always chooses an $x$-value $\hat{x}$ such that the $y$-values of points in $\L_{\hat{x}}$ are all at most $\frac{k}{2}$, and the $y$-values of points in $\R_{\hat{x}}$ are all greater than $\frac{k}{2}$.
By following this strategy from the first turn, player A will always be able to choose such an $\hat{x}$.
Let $\ell$ be the size of a longest up-run with $y$-values at most $\frac{k}{2}$, and let $r$ be the size of a longest up-run with $y$-values greater than $\frac{k}{2}$.
Due to player A's strategy, these two up-runs together form an $(\ell+r)$-up-run.
Applying Erd\H{o}s-Szekeres to each of the parts separately yields that there are at most $\ell\left\lfloor\frac{k}{2}\right\rfloor$ points with $y$-values at most $\frac{k}{2}$, and at most $r\left\lfloor\frac{k-1}{2}\right\rfloor$ points with $y$-values greater than $\frac{k}{2}$ at any time.
Plugging in $\ell+r=m-1$ at the penultimate turn yields the claimed upper bound.
\end{proof}

Since $B(m,k) \ge m$, Proposition~\ref{prop:BgameUB} implies that $B(m,2) = B(m,3) = m$ for all $m$.
For the lower bound, we first prove the following, weaker version of Proposition~\ref{prop:BgameLB2}.

\begin{proposition}\label{prop:BgameLB1}
\[
B(m,k) \ge \left\lfloor\frac{k}{2}\right\rfloor (m-k+1)+k-1.
\]
\end{proposition}

\begin{proof}
We describe a strategy for player B.
At every step, in response to player A choosing $\hat{x}$, player B plays any value in $\{1,\dots ,k-1\}$ such that any other point in the same tier is separated from $\hat{x}$ by at least $\left\lfloor \frac{k}{2}\right\rfloor -1$ points.
Now let $\S = \{ p_1, \dots, p_m \}$ be an $m$-up-run.
Since $\S$ can have at most $k-2$ points $p_i$ such that $y(p_{i+1}) \ne y(p_i)$, there are at least $(m-1)-(k-2)=m-k+1$ points $p_i$ such that $p_{i+1}$ is in the same tier.
Since points in the same tier must be separated by at least $\left\lfloor \frac{k}{2}\right\rfloor -1$ points, this accounts for $m+\left(\lfloor\frac{k}{2}\rfloor -1\right) (m-k+1)$ total points played.
\end{proof}

By avoiding the bottom tiers when player A chooses an $x$-value to the far left and avoiding the top tiers when player A chooses an $x$-value to the far right, player B can adjust the strategy given above to guarantee a few more points.

\begin{proposition}\label{prop:BgameLB2}
Let $k \ge 4$.
If $k$ is even, then $\displaystyle B(m,k) \ge \frac{k}{2} (m-k+5) -3$.\\
If $k$ is odd, then $\displaystyle B(m,k) \ge \frac{k-1}{2} (m-k+6) -3$.
\end{proposition}

\begin{proof}
We describe a strategy for player B.
First assume that $k$ is even.
Suppose player A chooses $\hat{x}$.
Now, let $c = \max\{1, \frac{k}{2} -|\L_{\hat{x}}| \}$ and let $d = \min\{ k-1, \frac{k}{2} + |\R_{\hat{x}}| \}$.
Then player B plays any value in $\{c,\dots ,d\}$ such that any other point in the same tier is separated from $\hat{x}$ by at least $\frac{k}{2} -1$ points.
Thus for any point $p$, we have both $|\L_{x(p)}| \ge \max\left\{ 0, \frac{k}{2}-y(p) \right\}$ and $|\R_{x(p)}| \ge \max\left\{ 0, y(p) - \frac{k}{2} \right\}$ at any time after $p$ has been played.
Now let $\S = \{ p_1, \dots, p_m \}$ be an $m$-up-run.
Observe that there are at least $(m-1)-\big( y(p_m) - y(p_1) \big)$ pairs of consecutive points in $\S$ which are in the same tier.
Then counting points in $\S$, points separating consecutive points in $\S$ in the same tier, and points to the left and right of $\S$ yields at least 
\[
m + \Big[ m-1 -\Big(y(p_m) - y(p_1)\Big) \Big] \left( \frac{k}{2} -1 \right) + \max\left\{ 0, \frac{k}{2}-y(p_1) \right\} + \max\left\{ 0, y(p_m) - \frac{k}{2} \right\}
\]
 total points.
The above expression is minimized when $y(p_m)-y(p_1) = k-2$, which means that at least $m + \big[ m-1 -(k-2) \big] \left( \frac{k}{2} -1 \right) + k-2 = \frac{k}{2} (m-k+5) -3$ points have been played.

If $k$ is odd, then player B can use the strategy for $k-1$ (by never choosing the top tier) to obtain the claimed bound.
\end{proof}

\section{Establishing an Upper Bound for $\ESO(m,3)$}\label{sec:m3UB}

We provide an upper bound for $\ESO(m,3)$ by describing a strategy for player A.
This strategy (see Definition~\ref{def:Astrategy}) will assess the state of the game and decide which of several strategy ``modes" to use at that time.
We begin by defining some terms used throughout this section and the next, and then define the different strategy modes.
We finish by analyzing the full strategy in Lemma~\ref{lemma:m3UB}.

\begin{definition}\textit{Columns, Rows, Notches}\\
Suppose a game of $A_{m,k}$ is underway with a set of points $\P = \{p_1, \dots, p_t\}$ having been played.
Let $p_0 = 0$ and $p_{t+1}=1$.
Then each of the intervals $\{\big(x(p_i), x(p_{i+1})\big)\}_{i=0}^t$ is a set of equivalent choices for player A's next turn, called a \textit{column}.
When we say that player A plays in a column $(x(p_i),x(p_{i+1}))$, we mean that player A chooses any $x$-value in $(x(p_i),x(p_{i+1}))$; for formality, without loss of generality we may assume that player A chooses the $x$-value $\frac{x(p_i)+x(p_{i+1})}{2}$.

Similarly, a \textit{row} is an interval between consecutive $y$-values of $\P$ from which all choices that player $B$ makes are equivalent given the state of the game.

A \textit{notch} is a unit of measurement between columns or between rows with respect to a point and a subset of points.
Let $p \in \P$, let $\A \subseteq \P$, and let $\tilde{\A} = \A \cup \{p\}$.
Let $p_1, \dots, p_s$ be an ordering of $\tilde{\A}$  where $p_i = p$.
Consider the columns induced by $\tilde{\A}$, say $I_j = (x(p_j),x(p_{j+1}))$ for $0 \le j \le t$.
Then for each $j < i$, we say that the column $I_j$ is $i-j-1$ \textit{notches left of $p$ with respect to $\A$}, or simply $i-j-1$ \textit{notches left of $p$} if $\A$ is understood.
For each $j \ge i$, we say that the column $I_j$ is $j-i$ \textit{notches right of $p$}.

Similarly, we speak of rows being some number of notches \textit{above} or \textit{below} $p$ (with respect to $\A$).
(See Figure~\ref{fig:notches}.)
\end{definition}

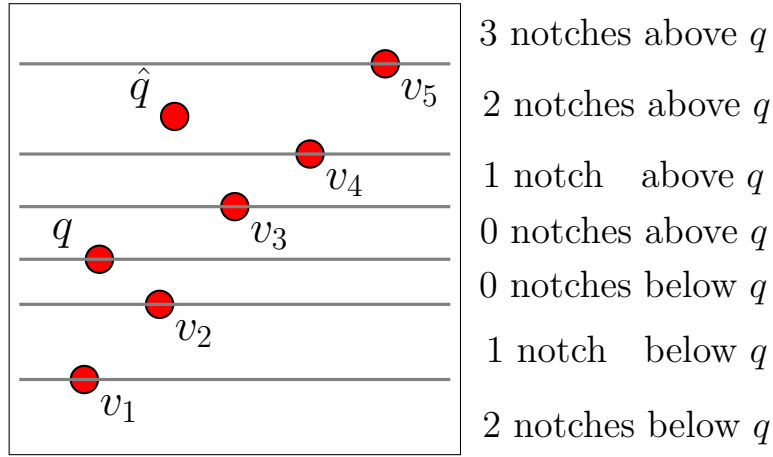
\begin{figure}[h!]
\center
\begin{tikzpicture}

	\draw (0,0) rectangle (6,6);
	\node (1a) at (0,1) {};
	\node (1b) at (6,1) {};
	\node (2a) at (0,2) {};
	\node (2b) at (6,2) {};
	\node (3a) at (0,3.3) {};
	\node (3b) at (6,3.3) {};
	\node (4a) at (0,4) {};
	\node (4b) at (6,4) {};
	\node (5a) at (0,5.2) {};
	\node (5b) at (6,5.2) {};
	\node (qa) at (0,2.6) {};
	\node (qb) at (6,2.6) {};
    
\begin{scope}[every node/.style={circle,thick,draw,shade,shading=axis,left color=red,right color=red,shading angle=90}]
    \node (1) at (1,1) {};
    \node (2) at (2,2) {};
    \node (3) at (3,3.3) {};
    \node (4) at (4,4) {};
    \node (5) at (5,5.2) {};
    \node (q) at (1.2,2.6) {};
    \node (qh) at (2.2,4.5) {};
\end{scope}

	\node at ([shift={(-40:.6)}]1) {\LARGE $v_1$};
	\node at ([shift={(-40:.6)}]2) {\LARGE $v_2$};
	\node at ([shift={(-40:.6)}]3) {\LARGE $v_3$};
	\node at ([shift={(-40:.6)}]4) {\LARGE $v_4$};
	\node at ([shift={(-40:.6)}]5) {\LARGE $v_5$};
	\node at ([shift={(145:.6)}]q) {\LARGE $q$};
	\node at ([shift={(140:.6)}]qh) {\LARGE $\hat{q}$};

	\node at ([shift={(10:2.2)}]5b) {\Large 3 notches above $q$};
	\node at ([shift={(16:2.3)}]4b) {\Large 2 notches above $q$};
	\node at ([shift={(10:2.2)}]3b) {\Large 1 notch~~ above $q$};
	\node at ([shift={(9:2.2)}]qb) {\Large 0 notches above $q$};
	\node at ([shift={(-10:2.2)}]qb) {\Large 0 notches below $q$};
	\node at ([shift={(-16:2.3)}]2b) {\Large 1 notch~~ below $q$};
	\node at ([shift={(-16:2.3)}]1b) {\Large 2 notches below $q$};

\begin{scope}[-,
              every edge/.style={draw=gray,very thick}]
    \path [-] (1a) edge node {} (1b);
    \path [-] (2a) edge node {} (2b);
    \path [-] (3a) edge node {} (3b);
    \path [-] (4a) edge node {} (4b);
    \path [-] (5a) edge node {} (5b);
    \path [-] (qa) edge node {} (qb);
\end{scope}

\end{tikzpicture}
\caption{Let $\A = \{v_1,v_2,v_3,v_4,v_5\}$.
Then $\hat{q}$ is 2 notches above $q$ (with respect to $\A$).}
\label{fig:notches}
\end{figure}

\begin{definition}\textit{``$f$-Middling" Strategy Mode}\\
In the game $A_{m,3}$, we have the initial conditions $\S = \N = \W = \varnothing$, $(a_x,b_x) = (a_y,b_y) = (0,1)$ and $\t = 0$, where $\S$ is the point set in the \textit{active segment} $(a_x,b_x) \times (a_y,b_y)$ forming an up-run, $\N$ is an up-run of \textit{saved} points previously in $\S$, $\W$ is the set of ``wasted" points not extending the up-run,  and \t ~is the number of times player B has chosen to not extend $\S$ into a longer up-run.
The parameter $f$ is a sequence $\{f_i\}_{i=1}^\infty$ of nonnegative integers given ahead of time.
As player A, do the following:
\begin{enumerate}
\item\label{Mstep:play} Play in a middlemost column of $\S$ with respect to $(a_x,b_x) \times (a_y,b_y)$, say the interval $( x(p_\ell), x(p_r) )$.
(Possibly with $p_\ell = (0,0)$ or $p_r = (1,1)$ being artificial points.)
Player B chooses a row, creating the new point $q$.
Assume the game is not over and thus that $y(q) \in (a_y,b_y)$.

\item If $q$ extended $\S$ to be a longer up-run, then return to Step~\ref{Mstep:play}.
Otherwise, $q$ is above $p_r$ or below $p_\ell$; increment \t.
If $|\NE(q) \cap \S| \le f_\t$ or $|\SW(q) \cap \S| \le f_\t$, then exit the strategy mode now.
Add $q$ to $\W$.
Go to Step~\ref{Mstep:pointright} if $q$ is above $p_r$ and go to Step~\ref{Mstep:pointleft} if $q$ is below $p_\ell$.

\item\label{Mstep:pointright} Let $\hat{q} = \argmax_{p \in \SE(q) \cap \S} x(p)$ and redefine $a_x := x(\hat{q})$ and $a_y := y(\hat{q})$.
Then update $\S$ and $\N$ by deleting $(\SW(\hat{q}) \cup \{\hat{q}\}) \cap \S$ from $\S$ and adding these points to $\N$.
Return to Step~\ref{Mstep:play}.

\item\label{Mstep:pointleft} Let $\hat{q} = \argmin_{p \in \NW(q) \cap \S} x(p)$ and redefine $b_x := x(\hat{q})$ and $b_y := y(\hat{q})$.
Then update $\S$ and $\N$ by deleting $(\NE(\hat{q}) \cup \{\hat{q}\}) \cap \S$ from $\S$ and adding these points to $\N$.
Return to Step~\ref{Mstep:play}.
\end{enumerate}

We illustrate the $f$-Middling strategy mode in the Appendix.
\end{definition}

\begin{definition}\textit{Barb, ``Playing the Barb" Strategy Mode}\\
An \textit{$(s, t)$-barb}, or simply \textit{barb}, is a subset $\B = \{w,z\} \cupdot \U \cupdot \V \subseteq \P$ such that $w,z$ form a down-run, $\U$ and $\V$ are (possibly empty) $s$- and $t$-up-runs, respectively, and also $x(w) < \max_{u \in \U} x(u) < \min_{v \in \V} x(v) < x(z)$ and $\max_{u \in \U} y(u) < y(z) < y(w) < \min_{v \in \V} y(v)$.
We say that $w$ and $z$ are the \textit{spikes} of $\B$, that $\U$ is the \textit{bottom wire} of $\B$, and that $\V$ is the \textit{top wire} of $\B$.
If one of $\U$ or $\V$ is empty, we drop the appropriate inequality conditions and say that $\B$ is a \textit{half-barb}.
We say that player A \textit{plays the barb} by doing the following:
\begin{enumerate}
\item\label{PBstep:play} Let $a := \max_{p \in \U \cup \{w\}} x(p)$ and $b := \min_{p \in \V \cup \{z\}} x(p)$.
Play in the column $(a,b)$ with respect to $\B$.
Player B chooses a row, creating the new point $q$.

\item Assuming the game is not over, $q$ is in the row zero notches above $w$ or zero notches below $z$ with respect to $\B$.
If the former, add $q$ to $\V$; if the latter, add $q$ to $\U$.
Return to Step~\ref{PBstep:play}.
\end{enumerate}

\end{definition}

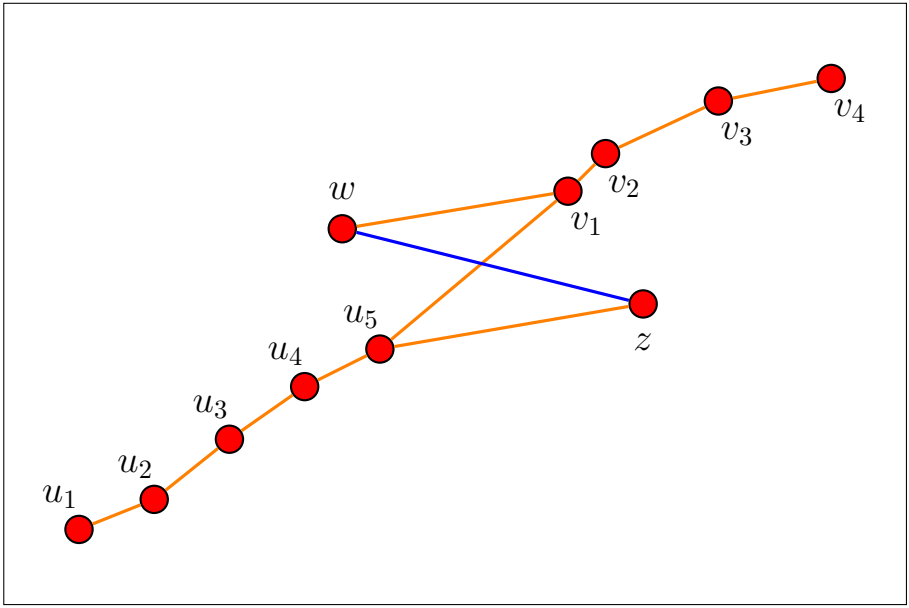
\begin{figure}[h!]
\center
\begin{tikzpicture}

	\draw (0,0) rectangle (12,8);
    
\begin{scope}[every node/.style={circle,thick,draw,shade,shading=axis,left color=red,right color=red,shading angle=90}]
    \node (u1) at (1,1) {};
    \node (u2) at (2,1.4) {};
    \node (u3) at (3,2.2) {};
    \node (u4) at (4,2.9) {};
    \node (u5) at (5,3.4) {};
    \node (v1) at (7.5,5.5) {};
    \node (v2) at (8,6) {};
    \node (v3) at (9.5,6.7) {};
    \node (v4) at (11,7) {};
    \node (w) at (4.5,5) {};
    \node (z) at (8.5,4) {};
\end{scope}

	\node at ([shift={(120:.5)}]u1) {\Large $u_1$};
	\node at ([shift={(120:.5)}]u2) {\Large $u_2$};
	\node at ([shift={(120:.5)}]u3) {\Large $u_3$};
	\node at ([shift={(120:.5)}]u4) {\Large $u_4$};
	\node at ([shift={(120:.5)}]u5) {\Large $u_5$};

	\node at ([shift={(-60:.5)}]v1) {\Large $v_1$};
	\node at ([shift={(-60:.5)}]v2) {\Large $v_2$};
	\node at ([shift={(-60:.5)}]v3) {\Large $v_3$};
	\node at ([shift={(-60:.5)}]v4) {\Large $v_4$};
	
	\node at ([shift={(90:.5)}]w) {\Large $w$};
	\node at ([shift={(-90:.5)}]z) {\Large $z$};

\begin{scope}[-,
              every edge/.style={draw=orange,very thick}]
    \path [-] (u1) edge node {} (u2);
    \path [-] (u2) edge node {} (u3);
    \path [-] (u3) edge node {} (u4);
    \path [-] (u4) edge node {} (u5);
    \path [-] (u5) edge node {} (z);
    
    \path [-] (u5) edge node {} (v1);
    
    \path [-] (w) edge node {} (v1);
    \path [-] (v1) edge node {} (v2);
    \path [-] (v2) edge node {} (v3);
    \path [-] (v3) edge node {} (v4);
\end{scope}

\begin{scope}[-,
              every edge/.style={draw=blue,very thick}]
    \path [-] (w) edge node {} (z);
\end{scope}

\end{tikzpicture}
\caption{A (5,4)-barb, with $\U = \{u_1,u_2,u_3,u_4,u_5\}$, $\V = \{v_1,v_2,v_3,v_4\}$, and spikes $w,z$.}
\end{figure}

The following observation illustrates the usefulness of barbs for player A.

\begin{observation}
If a game of $A_{m,3}$ is underway with an $(s,t)$-barb $\B$, then the game will end after at most $m - s -t$ more turns if player A plays the barb $\B$.
\end{observation}

\begin{definition}\textit{``$w$-Barb" Strategy Mode}\\
Suppose a game of $A_{m,3}$ is underway with the point set $\P = \U \cupdot \{r_1,q_1\} \cupdot \V_1$ such that $\B_1 = \{r_1,q_1\} \cup \V_1$ is a half-barb with spikes $r_1,q_1$ (where $x(r_1) < x(q_1)$) and top wire $\V_1$, such that $\U = \SW(r_1)$, and such that $\U$ is an up-run.
Let $\hat{r}_1 = \argmin_{p \in \NW(q_1)\setminus\{r_1\}} y(p)$.
The parameter $w$ is given as a positive integer.
As player A, do the following:
\begin{enumerate}
\item\label{Bstep:play} Let $i$ be maximum such that $\B_i$ has been defined; play in the column $(x(r_i),x(\hat{r}_i))$.
Player B chooses a row, creating the new point $q$.
Assuming the game is not over, we have $y(q) < y(\hat{r}_i)$.

\item If $q$ is zero notches above $r_i$ with respect to $\B_i$, add $q$ to $\mathcal{V}_i$, then redefine $\hat{r}_i := q$ and go to Step~\ref{Bstep:play}.
Otherwise, $q$ is below $r_i$, say by $d$ notches with respect to $\U$.
If $d \ge w - i$ and $d > 0$, then go to Step~\ref{Bstep:stepdown}.
Otherwise, $d < w - i$ or $d = 0$; go to Step~\ref{Bstep:playbarb}.

\item\label{Bstep:stepdown} Consider $q_i$ a lost point.
Rename $q$ to be $q_{i+1}$, define $r_{i+1} := \argmin_{p \in \NW(q_{i+1})} y(p)$, define $\hat{r}_{i+1} := \argmin_{p \in \NW(q_{i+1})\setminus\{r_{i+1}\}} y(p)$, define $\V_{i+1} : = \V_i \cupdot \NW(q_{i+1}) \setminus \{ r_{i+1} \}$, and define the half-barb $\B_{i+1} := \{r_{i+1},q_{i+1}\} \cupdot \V_{i+1}$.
Return to Step~\ref{Bstep:play}.

\item\label{Bstep:playbarb} Consider $\NW(q)$ as lost points.
Let $\hat{\U} = \SW(q) \cup \{q\}$ and let $\B = \{r_i,q_i\} \cup \hat{\U} \cup \V_i$.
Then $\B$ is a barb with spikes $r_i$ and $q_i$, top wire $\V_i$ and bottom wire $\hat{\U}$; play the barb $\B$ until the game finishes.
\end{enumerate}

We illustrate a step of the $w$-Barb strategy mode in the Appendix.%
\footnote{One can replace the rules for going to Steps~\ref{Bstep:stepdown}~and~\ref{Bstep:playbarb} with rules that compare $d$ to $|\SW(q)|$ rather than comparing $d$ to $w-i$.
This would not improve the result of Lemma~\ref{lemma:m3barbbound}, but would be a better strategy for player~A when player~B does not play optimally.}

\end{definition}

\begin{lemma}\label{lemma:m3barbbound}
If player A adopts the $w$-Barb strategy mode in a game of $A_{m,3}$ with $w \ge \left\lfloor\sqrt{\max\{2|\U|-\frac{15}{4},0\}}+\frac32\right\rfloor +1$, then the game ends after a total of at most $m+w$ turns.
\end{lemma}

\begin{proof}
Since the game would end just as soon if a 3-down-run is formed, we may assume that the game ends once an $m$-up-run is formed.
Observe that if the game ends after reaching Step~\ref{Bstep:playbarb}, then the $d+1$ points in $\NW(q)$ and the $i$ points $q_1, \dots, q_i$ are the only points not in the $m$-up-run containing $\V_i \cup \hat{\U}$.
Thus a total of $m+d+i+1$ turns were taken.
Since $d+1 \le w-i$, at most $m+w$ turns were taken.

If the game ends without ever reaching Step~\ref{Bstep:playbarb}, then it is because $\V_i \cup \U$ is an $m$-up-run containing all points except $q_1, \dots, q_i$, making a total of $m+i$ turns taken.
Now observe that for any $2 \le j < w$ we have $|\NW(q_j) \setminus \NW(q_{j-1})| \ge w-j$ and also that $\mathop{\cdot\hspace*{-7.5pt}\bigcup}_{j=2}^i \Big( \NW(q_j) \setminus \NW(q_{j-1}) \Big) = \left(\bigcup_{j=2}^i \NW(q_j)\right) \setminus \NW(q_{j-1}) \subseteq \U$ for all $i \ge 2$.
Hence if $i = w+1$, then $|\U| \ge \sum_{j=2}^i \Big| \NW(q_j) \setminus \NW(q_{j-1}) \Big| \ge 1+1+\sum_{j=2}^{w-1} (w-j) = \frac{(w-\frac{3}{2})^2 +\frac{15}{4}}{2}$.
Since $w > \sqrt{\max\{2|\U|-\frac{15}{4},0\}}+\frac32$, this implies $|\U| \ge\frac{(w-\frac{3}{2})^2 +\frac{15}{4}}{2} > |\U|$, a contradiction.
Thus $i \le w$ always, and therefore at most $m+w$ turns are taken in this case as well.
\end{proof}

\begin{definition}\label{def:Astrategy}\textit{$f$-Combined Strategy}\\
Let $f$ be given.
As player A in the game $A_{m,3}$, do the following:
\begin{enumerate}
\item\label{Cstep:middling} (Middling Mode) Begin by playing the $f$-Middling strategy mode.
Play until the game ends or until the strategy mode terminates, making note of $\N$.

\item\label{Cstep:transitiona} (Transition) If the $f$-Middling mode strategy mode terminated, then in the active segment $(a_x,b_x) \times (a_y,b_y)$ we have an up-run $\S$ and a point $q$ in a middlemost column $(x(p_\ell),x(p_r))$ with respect to $\S$, where $|\NE(q) \cap \S| \le f_\t$ or $|\SW(q) \cap \S| \le f_\t$ and without loss of generality we may assume $y(q) < y(p_\ell)$.\\
If $q$ is at least one notch below $p_\ell$ with respect to $\S$, then let $\U := \SW(q) \cap \S$, $r_1 := \argmin_{p \in \NW(q) \cap \S} y(p)$, $q_1 := q$, $\V_1 := \NE(r_1) \cap \S$ and go to Step~\ref{Cstep:barb}.\\
Else, $q$ is zero notches below $p_\ell$ with respect to $\S$; proceed to Step~\ref{Cstep:transitionb}.

\item\label{Cstep:transitionb} (Transition) Play in the column $(x(p_\ell),x(q))$.
Player B chooses a row, creating the new point $\hat{q}$.
Assuming the game is not over, we have $a_y < y(\hat{q}) < b_y$ and may assume without loss of generality that either $\hat{q}$ is zero notches above $p_\ell$ with respect to $\S$ or at least one notch below $q$ with respect to $\S$.\\
If the former, then let $\U := \SW(p_\ell) \cap \S$, $r_1 := p_\ell$, $q_1 := q$, $\V_1 := \NE(p_\ell) \cap \S$. Also add $\hat{q}$ to $\V_1$. Go to Step~\ref{Cstep:barb}.\\
If the latter, then $\hat{q}$ is say $k$ notches below $q$ with respect to $\S$.
Then let $\U := \SW(\hat{q}) \cap \S$, $r_1 := \argmin_{p \in \NW(\hat{q}) \cap \S} y(p)$, $q_1 := \hat{q}$, and $\V_1 := \NE(r_1) \cap \S$.
Also add $q$ to $\W$.
Go to Step~\ref{Cstep:barb}.

\item\label{Cstep:barb} (Barb Mode) Given $\U$, $r_1$, $q_1$ and $\V_1$, play $A_{m-|\N|,3}$ by adopting the $w$-Barb strategy mode in $(a_x,b_x) \times (a_y,b_y)$ with $w = \left\lfloor\sqrt{\max\{2|\U|-\frac{15}{4},0\}}+\frac32\right\rfloor +1$ until the game ends.
\end{enumerate}

\end{definition}

\begin{lemma}\label{lemma:m3UB}
Suppose player A uses the $f$-Combined Strategy in the game $A_{m,3}$ with $f=\{f_i\}_{i=1}^\infty$.
Suppose also that there exists $T \ge 2$ meeting the following conditions:
\begin{enumerate}
\item\label{cond:transition} $\sum_{i=1}^T (f_i+2) +(f_T+1) \ge m$, and

\item\label{cond:barb} $i + \left\lfloor\sqrt{\max\{2|\U|-\frac{15}{4},0\}}+\frac32\right\rfloor \le T$ for all $i \in [T-1]$.
\end{enumerate}

Then the game ends after a total of at most $m+T+1$ turns.
\end{lemma}

\begin{proof}
Suppose the game has ended, again assuming that it was not because a 3-down-run was formed.
We proceed by case analysis.

Suppose the game ended while still in Step~\ref{Cstep:middling}.
This means that each time $\t$ was incremented in the $f$-Middling strategy mode (say to $j$), at least $f_j+2$ points were added to $\N$ from $\S$ ($\SW(\hat{q}) \cup \{\hat{q}\}$ or $\NE(\hat{q}) \cup \{\hat{q}\}$) and at least $f_j+1$ points were kept in $\S$.
Since there was no $m$-up-run at this time, we have $\sum_{i=1}^j (f_i+2) +(f_j+1) \le |\N|+|\S| < m$.
Since this is true for all $j \in [\t]$ but not for $j=T$ (by Condition~\ref{cond:transition}), we have $\t < T$.
Since the points played in the game were either in $\N \cup \S$ which now form an $m$-up-run or the $\t$ points in $\W$, we have $m+\t$ points played.
Thus at most $m+T-1$ turns were taken.

It is also possible that the game ended while in Steps~\ref{Cstep:transitiona}~and~\ref{Cstep:transitionb}, specifically when the last point $\hat{q}$ was played in the ``former" case of Step~\ref{Cstep:transitionb}.
Again, the total number of points played is $m+\t$.
Observe that if Step~\ref{Cstep:transitiona} is reached, then $\t \le T$.
This is because $\sum_{i=1}^{\t-1} (f_i+2) +(f_{\t-1}+1) < m$, which cannot be true for $\t = T+1$ by Condition~\ref{cond:transition}.
Thus at most $m+T$ turns were taken.

Finally, suppose the game ended while in Step~\ref{Cstep:barb}.
First note that all points played outside of $(a_x,b_x) \times (a_y,b_y)$ were played before reaching Step~\ref{Cstep:transitiona} and totaled $|\N|+|\W|$, where $|\W|=\t-1$ because we exited the $f$-Middling strategy mode.
All points played within $(a_x,b_x) \times (a_y,b_y)$ were either a single point added to $\W$ in Step~\ref{Cstep:transitionb} or part of the subgame $A_{m-|\N|,3}$ in Step~\ref{Cstep:barb}.
By Lemma~\ref{lemma:m3barbbound}, the subgame took at most $m-|\N|+ \left\lfloor\sqrt{\max\{2f_\t -\frac{15}{4},0\}}+\frac32\right\rfloor +1$ turns since we had $|\U| \le f_\t$ upon entering Step~\ref{Cstep:barb}.
Altogether, then, $(|\N|+\t-1)+(1)+(m-|\N|+\left\lfloor\sqrt{\max\{2f_\t -\frac{15}{4},0\}}+\frac32\right\rfloor +1) = m+\t+\left\lfloor\sqrt{\max\{2f_\t -\frac{15}{4},0\}}+\frac32\right\rfloor +1$ points were played.
As noted in a previous case, $\t < T$ because we reached Step~\ref{Cstep:transitiona}.
Thus by Condition~\ref{cond:barb} at most $m + T +1$ turns were taken.
\end{proof}

\begin{corollary}\label{cor:m3UB}
Let $m \ge 3$ be given, and let $T$ be the least integer satisfying the inequality $(T-1)^3 +17(T-1) \ge 6(m-3)$.
Then $\ESO(m,3) \le m + T +1$.
\end{corollary}

\begin{proof}
Define $f = \{f_i\}_{i=1}^\infty$ by $f_i = \frac{(T-\frac12-i)^2+\frac74}{2}$ for $i \le T-1$ and by $f_i = 0$ for $i \ge T$.
Observe that $f$ satisfies Condition~\ref{cond:barb} from Lemma~\ref{lemma:m3UB} by definition since the the first $T-1$ terms are the maximum integers satisfying the inequality condition.
To see that $f$ satisfies Condition~\ref{cond:transition}, it suffices to observe that $\sum_{i=1}^{T-1} (f_i+2) = \frac16 \left[ (T-1)^3 +17(T-1) \right]$.
This implies $\sum_{i=1}^T (f_i+2) +(f_T +1) = \frac16 \left[ (T-1)^3 +17(T-1) \right] +3 \ge m$.
\end{proof}

Applying $T = \lceil \sqrt[3]{6(m-3)} \rceil +1$ to Corollary~\ref{cor:m3UB} yields 
\[
\ESO(m,3) \le m + \lceil \sqrt[3]{6(m-3)} \rceil +2 < m + \sqrt[3]{6(m-3)} +3 < m + \sqrt[3]{6m} +3.
\]

\section{Establishing a Tighter Lower Bound for $\ESO(m,3)$}\label{sec:m3LB}

To provide a lower bound for $\ESO(m,3)$, we now give a strategy for player B.
First, we provide some technical definitions.

\begin{definition}\textit{$h(x)$, $h^{\uparrow n}(x)$, $h_{\downarrow n}(x)$ given $\C$, $c$, and $d$}\\
Consider a real number $x \in (0,1)$, an integer $n$, an up-run $\C$ with distinct $y$-values, and real numbers $0 \le c < d \le 1$.

We define $h(x)$ to be a $y$-value such that $c<y<d$ and such that $\C \cup\{(x,y)\}$ is an up-run with distinct $y$-values.
Let $\C = \{ p_1, \dots, p_t \}$ and let $p_0 = (0,0)$ and $p_{t+1} = (1,1)$.
Let $j = \max\{ i : x(p_i) < x\}$.
If $c \ge y(p_{j+1})$ or $d \le y(p_j)$, then $h(x)$ is undefined.
Otherwise, let $h(x) = \max\big\{c, y(p_j)\big\} + \big(x-x(p_j)\big) \dfrac{\min\big\{d, y(p_{j+1})\big\} -y(p_j)}{x(p_{j+1}) - \max\big\{c, y(p_j)\big\} }$.

We define $h^{\uparrow n}(x)$ to be a $y$-value $n$ notches above $\big( x, h(x) \big)$ with respect to $\C$ such that $c \le y < d$.
Let $(a,b)$ be the row $\max\{0,n\}$ notches above $\big( x, h(x) \big)$ with respect to $\C$.
If this row does not exist or if $a \ge d$, then $h^{\uparrow n}(x)$ is undefined.
Otherwise, we define $h^{\uparrow n}(x) = \max \left\{ c, ~ \dfrac{a + \min\{b,d\} }{2} \right\}$.

Let $h_{\downarrow n}(x)$ be defined similarly, but to be a $y$-value $n$ notches below $\big( x, h(x) \big)$ with respect to $\C$ such that $c < y \le d$.
\end{definition}

\begin{definition}\textit{$w$-Fracturing Strategy}\\
Let $w$ be a given positive integer.
In a game of $A_{m,3}$, follow the steps below as player B.
Player B will decide whether to play on the line $y=x$ to extend a central up-run $\C$ or to play above (\textit{fracturing left}) or below (\textit{fracturing right}) $\C$.
When doing the latter, player B creates \textit{wastebins} $\W_j$ and corresponding intervals $W_j$.
The sets $\B_j$ are the \textit{banked up-runs}, the portions of the central up-run $\C$ which become fixed when player B fractures left or right.
The $x$-values $a \le a' \le a'' \le b'' \le b' \le b$ are used to decide whether to add to $\C$, create a new wastebin, or add to an existing wastebin.
The $y$-values $c<d$ are used to ensure the wastebins form an up-run.
The function $\varphi$ encodes the decision that player B makes.

\begin{enumerate}
\setcounter{enumi}{-1}

\item (Initialize)
Define $z_j = \frac{(w-j)(w-j+1)}{2} +1$ for $1 \le j \le w-1$ and $z_j = 1$ for $j \ge w$.
Let $a = a' = a'' = 0 = c$, let $b'' = b' = b = 1 = d$, and let $\C = \C_1 = \varnothing$.
Define $\varphi(x) := x$ for all $x \in (0,1)$, and let $i = 1$.

\item\label{Fstep:nextmove} (Choose $y$)
Assuming the game is not over, player A chooses an $x$-value, say $\hat{x}$.
Choose the $y$-value $\varphi(\hat{x})$ to make the point $q$.
If $\hat{x} < a$ or $b < \hat{x}$, then go to Step~\ref{Fstep:previousfracture}.
If $a < \hat{x} < a''$, then go to Step~\ref{Fstep:leftfracture}.
If $b'' < \hat{x} < b$, then go to Step~\ref{Fstep:rightfracture}.
Else, $a'' < \hat{x} < b''$; add $q$ to $\C$ and $\C_i$ and go to Step~\ref{Fstep:update2}.

\item\label{Fstep:previousfracture} (Previous fracture)
In this case, we must have $\hat{x} \in W_j$ for some $j \in [i-1]$.
Add $q$ to $\W_j$ and return to Step~\ref{Fstep:nextmove}.

\item\label{Fstep:leftfracture} (Fracture left)
Let $\ell = \min_{p \in \SE(q)} x(p)$ and $r = \max_{p \in \SE(q)} x(p)$.
Let $W_i = (a, \ell )$, let $\W_i = \{q\}$, and let $\B_i = \{ p \in \C_i : x(p) \le r \}$.
Then redefine $a = \ell$, $a' = r$, and $c = y(q)$.
Go to Step~\ref{Fstep:update1}.

\item\label{Fstep:rightfracture} (Fracture right)
Let $\ell = \min_{p \in \NW(q)} x(p)$ and $r = \max_{p \in \NW(q)} x(p)$.
Let $W_i = (r, b )$, let $\W_i = \{q\}$, and let $\B_i = \{ p \in \C_i : \ell \le x(p) \}$.
Then redefine $b' = \ell$, $b = r$, and $d = y(q)$.
Go to Step~\ref{Fstep:update1}.

\item\label{Fstep:update1} (Begin updating)
Increment $i$.
Define $\C_i = \C_{i-1} \setminus \B_{i-1}$.
Go to Step~\ref{Fstep:update2}.

\item\label{Fstep:update2} (Finish updating)
Redefine $a''$ to be the maximum value of $x(p)$ from $p \in \C_i$ such that $|\NE(p) \cap \C_i| \ge z_i -1$, or to be $a'$ if there is no such $p$.
Similarly, redefine $b''$ to be the minimum value of $x(p)$ from $p \in \C_i$ such that $|\SW(p) \cap \C_i| \ge z_i -1$, or to be $b'$ if there is no such $p$.
Now, we redefine $\varphi$ on $(a,b)$:
\[
\varphi(x) = 
\begin{cases}
h^{\uparrow (w-i+1)}(x), & a < x < a''\\
h(x), & a'' < x < b''\\
h_{\downarrow (w-i+1)}(x), & b'' < x < b.
\end{cases}
\]
The definition of $\varphi$ on $(0,a) \cup (b,1)$ remains unchanged.
Return to Step~\ref{Fstep:nextmove}.

\end{enumerate}

\end{definition}

\begin{lemma}\label{lem:fracturing}
If player B uses the $w$-Fracturing Strategy in a game of $A_{m,k}$, then we have the following:
{ \renewcommand\labelenumi{(\theenumi)}
\begin{enumerate}
\item\label{lem:fracturing:CUB} $|\C_j| \le 2z_j -1$ always for all $j$.
\item\label{lem:fracturing:BUB1} If $j<w$ and $\B_j$ is defined when $a < \hat{x} < a'$ or $b' < \hat{x} < b$, then $|\B_j| \le w-j$.
\item\label{lem:fracturing:BUB2} If $j<w$ and $\B_j$ is defined when $a' < \hat{x} < a''$ or $b'' < \hat{x} < b'$, then $|\B_j| \le z_j +w -j$.
\item\label{lem:fracturing:CLB} After redefining $a''$ and $b''$ in Step~\ref{Fstep:update2}, if $i \le w$ and $a < a''$ or $b'' < b$ then $|\C_i| \ge z_i$.
\item\label{lem:fracturing:WellDefined} The $w$-Fracturing Strategy is well-defined.
\item\label{lem:fracturing:No3down} The game does not end with a 3-down-run.
\end{enumerate}
}
\end{lemma}

\begin{proof}
(\ref{lem:fracturing:CUB})
Every turn, $\C_i$ is newly defined or increases by at most one.
If $|\C_i| = 2z_i-1$, then $a'' = b''$ and no more points can be added to $\C_i$.
Once $\C_{j+1}$ is defined, the set $\C_j$ never changes.

(\ref{lem:fracturing:BUB1})
Suppose without loss of generality that $\B_j$ is defined when $a < \hat{x} < a'$.
Since $\hat{x} < a'$, all points in $\B_j$ are to the right of $q$, and there is at least one point in $\SE(q) \cap \C$ in $\B_{j'}$ for some $j' < j$.
Since $q$ was placed $w-j+1$ notches above $\big( x, h(x) \big)$, we have $|\B_j| = |\SE(q) \cap \C_j| \le (w-j+1)-1 = w-j$.

(\ref{lem:fracturing:BUB2})
Suppose without loss of generality that $\B_j$ is defined when $a' < \hat{x} < a''$.
Since $a' < \hat{x}$, there may be some points in $\B_j$ to the left of $q$; by (\ref{lem:fracturing:CUB}) there are at most $z_j -1$ of these.
Together with the $w-j+1$ points in $\B_j$ to the right of $q$, we have $|\B_j| \le (z_j -1) + (w-j+1) = z_j +w-j$.

(\ref{lem:fracturing:CLB})
If $a' < a''$ or $b'' < b'$, then $|\C_i| \ge z_i$ by redefinition of $a''$ and $b''$.
We now proceed by induction on $i$.
When $i=1$, we must have $a' < a''$ or $b'' < b'$ since $a=a'=0$ and $b'=b=1$.
So suppose $2 \le i \le w$ and that $|\C_j| \ge z_j$ for all $j <i$.
Since the size of $\C_i$ never decreases while $i$ is constant, suppose that $\C_i$ has just been defined in Step~\ref{Fstep:update1} after the point $q$ has been played.
If $a' < x(q) < a''$ or $b'' < x(q) < b'$, then we are done; suppose that $a < x(q) < a'$ or $b' < x(q) < b$.
Then by (\ref{lem:fracturing:BUB1}) we have $|\C_i| = |\C_{i-1} \setminus \B_{i-1}| \ge z_{i-1} - [w-(i-1)] = z_i$.

(\ref{lem:fracturing:WellDefined})
It suffices to show that in Step~\ref{Fstep:update2}, $h^{\uparrow(w-i+1)} (x)$ is defined for all $a < x < a''$.
If $a<x<a'$, then $h^{\uparrow(w-i+1)} (x)$ is defined if $|\C_i| \ge w-i$.
Since $z_j > w-j$ for all $j$, this follows from (\ref{lem:fracturing:CLB}).
If $a'<x<a''$, then $h^{\uparrow(w-i+1)} (x)$ is defined if $|\NE\big(x, h(x)\big) \cap \C_i| \ge w-i+1$.
Since $z_j \ge w-j+1$ for all $j$, this follows from the definition of $a''$.

(\ref{lem:fracturing:No3down})
We show that $\P$ can be partitioned into two up-runs $\C$ and $\bigcup_{j=1}^i \W_j$.
Since $\varphi$ is nondecreasing on $W_j$ when $W_j$ is defined and is not redefined on $W_j$, each wastebin $\W_j$ is an up-run.
The use of $c$ and $d$ in the definitions of $h^{\uparrow n} (x)$ and $h_{\downarrow n} (x)$, including the fact that $c < d$ always, guarantees that $\bigcup_{j=1}^i \W_j$ is an up-run.
\end{proof}

\begin{proposition}\label{prop:lowerbound}
If player B adopts the $w$-Fracturing Strategy in a game of $A_{m,3}$ where $(w+1)^3 -(w+1) < 6m$, then at least $m+w$ total points will be played.
\end{proposition}

\begin{proof}
Suppose we are at the end of such a game.
Then by Lemma~\ref{lem:fracturing}(\ref{lem:fracturing:No3down}), there is some $m$-up-run $\U$.
Either $\U$ intersects some wastebin set $\W_j$, or else $\U = \C$.

\textbf{Case 1:}  $\U$ intersects $\W_j$ for some $j \in [i-1]$.
Let $q_j \in \U \cap \W_j$.
None of the points in $\NW(q_j) \cup \SE(q_j) \subseteq \C$ can be in $\U$; by virtue of $q_j$ being added to $\W_j$, we know that $|\NW(q_j) \cup \SE(q_j)| \ge w-j+1$.
Now, for each $\alpha \in [j-1]$, select some $q_\alpha \in \W_\alpha$ and let $p_\alpha = \argmin_{p \in \NW(q_\alpha) \cup \SE(q_\alpha)} |x(p) - x(q_\alpha)|$.
From each pair $\{q_\alpha, p_\alpha\}$, at most one point can be in $\U$.
Observe that for each pair of sets $\W_\alpha, \W_{\alpha'}$ with $\alpha, \alpha' \in [i-1]$, there is a point $\hat{p} \in \C$ such that $x(q) < x(\hat{p}) < x(q')$ or $x(q') < x(\hat{p}) < x(q)$ for all $q \in W_\alpha$ and $q' \in W_{\alpha'}$.
Hence $p_1, \dots, p_{j-1}$ are distinct points in $\C$ (not necessarily ordered by $x$-values), and none of them are in $\NW(q_j) \cup \SE(q_j)$.
Thus there are at least $j-1$ points in $\{q_1,p_1, \dots, q_{j-1},p_{j-1} \}$ that are not in $\U$.
Counting the $m$ points in $\U$ and at least $(w-j+1)+(j-1)$ points we know not to be in $\U$, we have at least $m+w$ points total.

\textbf{Case 2:}  $\U = \C$.
Since the wastebin sets $\W_j$ have been defined for all $j \in [i-1]$, are nonempty, and do not intersect $\C$, we show that $i-1 \ge w$.
Since in this case the game cannot be over until $m = |\C| = \left|\left(\bigcup_{j \in [i-1]} \B_j \right) \cup \C_i\right|$, we suppose $i=w$ and show that $\left|\left(\bigcup_{j \in [w-1]} \B_j \right) \cup \C_w\right| < m$.
This implies $i > w$.
By Lemma~\ref{lem:fracturing}(\ref{lem:fracturing:BUB2}) we have $|\B_j| \le z_j + w  -j$ for all $j \in [w-1]$.
Also, by Lemma~\ref{lem:fracturing}(\ref{lem:fracturing:CUB}) we have that $|\C_w| \le 2z_w -1$.
Hence $\left|\left(\bigcup_{j \in [w-1]} \B_j \right) \cup \C_w\right| \le \sum_{j=1}^{w-1} (z_j + w - j) + 2z_w -1 = \frac16 \left[(w+1)^3 -(w+1)\right] <m$.
\end{proof}

Applying $w = \lfloor \sqrt[3]{6m} \rfloor -1$ to Proposition~\ref{prop:lowerbound} yields 
\[
\ESO(m,3) \ge m + \lfloor \sqrt[3]{6m} \rfloor -1 > m + \sqrt[3]{6m} -2.
\]

\section{Concluding Remarks}\label{sec:conclusion}

As already shown in Section~\ref{sec:mkLB}, one may imagine variations on the Erd\H{o}s-Szekeres on-line game.
For example, the game $B_{m,k}$ could be generalized so that player B plays from the set $[s]$.
Another variation in which player B plays from $[s]$ would restrict the game to playing $s$ points total and player B cannot repeat previous choices from $[s]$.
What is the maximum length up-run or down-run that player A can force in the $s$ turns?

Alternatively, one could consider additional target patterns.
For example, consider the $(s,t)$-$\V$ configuration:
a point $p$ which is the left-most point of an $s$-up-run and a $t$-down-run, or the right-most point of  an $s$-up-run and a $t$-down-run.
What is the minimum number of moves in which player A can force an $m$-up-run, a $k$-down run, or an $(s,t)$-$\V$?
When $s$ and $t$ are relatively smaller than $m$ and $k$, respectively, then this additional target pattern is likely to disrupt strategies which would be optimal for player B in a game of $A_{m,k}$.

Another possibility is to generalize to more than two dimensions for the board space.
In $R^n$, say that two points $(a_1, \dots, a_n)$ and $(b_1, \dots, b_n)$ are increasing if $a_i \le b_i$ for all $i \in [n]$ or $b_i \le a_i$ for all $i \in [n]$, and say that they are decreasing otherwise.
Let a set of $m$ pairwise increasing points be called a chain and a set of $k$ pairwise decreasing points be called a $k$-anti-chain.
For $s<n$, let player A choose the first $s$ coordinates of a point, and then player B choose the last $n-s$ coordinates of the point.
In how many turns can player A force an $m$-chain or a $k$-anti-chain?

Variations of the Erd\H{o}s-Szekeres on-line game might not have symmetry.
For example, $B(m,k)$ and $B(k,m)$ are not necessarily equal.
While $B(m,2) = m$, Proposition~\ref{prop:BgameUB} implies $B(2,m) \le \lfloor \frac{m}{2}\rfloor +1$.

Nevertheless, we believe that $\ESO(m,k)$ and $B(m,k)$ are very closely related, encapsulated in the following conjecture, which is stronger than Conjecture~\ref{con} mentioned in the introduction.
\begin{conjecture}
For all $m$ and $k$, $|\ESO(m,k)-B(m,k)|=o(mk)$.
\end{conjecture}

\section*{Acknowledgments}
This collaboration began as part of the 2016 Rocky Mountain--Great Plains Graduate Research Workshop in Combinatorics, supported in part by NSF-DMS Grants \#1604458, \#1604773, \#1604697 and \#1603823.
The authors are grateful for the workshop, which provided the opportunity for fruitful conversations.
The authors thank Gavin King, a workshop participant, for posing the ``Online Happy Ending Problem'' which inspired this project.

\bibliographystyle{abbrv}
\bibliography{ESO}

\newpage
\section*{Appendix}

\noindent \textbf{Example 1}

Now we give an example of how a game of $A_{m,3}$ might proceed when player A uses the $f$-Middling strategy mode, with $f = (3,2,1,0,0,0,\dots)$.
In the figures that follow, the labels of each point indicate the order in which they were played.

Until player B deviates from $\S$, we have $\t=0$ and the points look something as follows:

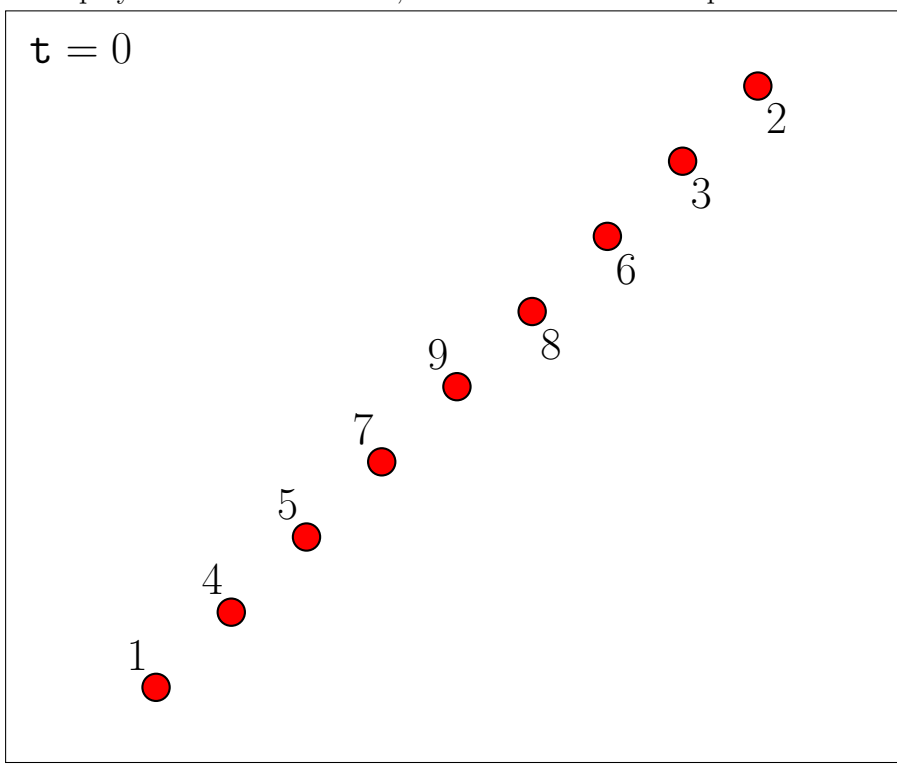
\begin{tikzpicture}

	\draw (0,0) rectangle (12,10);
	\node at (1,9.5) {\LARGE $\t = 0$};
    
\begin{scope}[every node/.style={circle,thick,draw,shade,shading=axis,left color=red,right color=red,shading angle=90}]
    \node (1) at (2,1) {};
    \node (4) at (3,2) {};
    \node (5) at (4,3) {};
    \node (7) at (5,4) {};
    \node (9) at (6,5) {};
    \node (8) at (7,6) {};
    \node (6) at (8,7) {};
    \node (3) at (9,8) {};
    \node (2) at (10,9) {};
\end{scope}

	\node at ([shift={(120:.5)}]1) {\LARGE 1};
	\node at ([shift={(-60:.5)}]2) {\LARGE 2};
	\node at ([shift={(-60:.5)}]3) {\LARGE 3};
	\node at ([shift={(120:.5)}]4) {\LARGE 4};
	\node at ([shift={(120:.5)}]5) {\LARGE 5};
	\node at ([shift={(-60:.5)}]6) {\LARGE 6};
	\node at ([shift={(120:.5)}]7) {\LARGE 7};
	\node at ([shift={(-60:.5)}]8) {\LARGE 8};
	\node at ([shift={(120:.5)}]9) {\LARGE 9};

\end{tikzpicture}\\
At this point we have $(a_x,b_x) \times (a_y,b_y) = (0,1) \times (0,1)$, $\S = \{1,2,3,4,5,6,7,8,9\}$, and $\N = \W = \varnothing$.

Suppose that on the next turn, player B deviates from $\S$ to form the point 10:

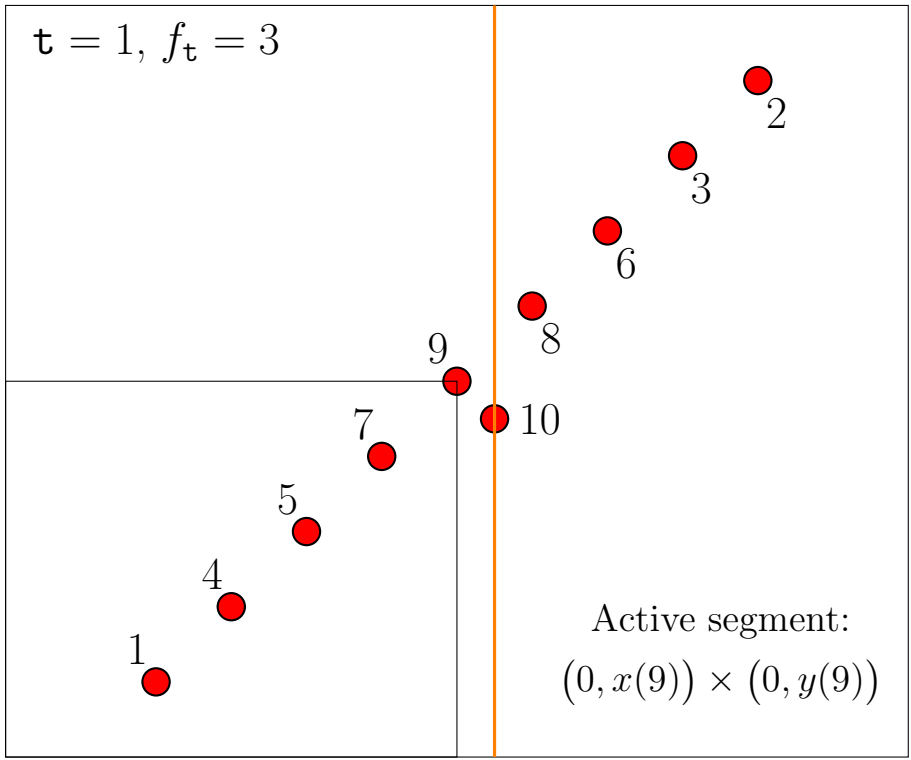
\begin{tikzpicture}

	\draw (0,0) rectangle (12,10);
	\node at (2,9.5) {\LARGE $\t = 1$, $f_\t = 3$};
	\node at (9.5,1.8) {\Large Active segment:};
	\node at (9.5,1) {\Large $\big( 0,x(9) \big) \times \big( 0,y(9) \big)$};
    
\begin{scope}[every node/.style={circle,thick,draw,shade,shading=axis,left color=red,right color=red,shading angle=90}]
    \node (1) at (2,1) {};
    \node (4) at (3,2) {};
    \node (5) at (4,3) {};
    \node (7) at (5,4) {};
    \node (9) at (6,5) {};
    \node (10) at (6.5,4.5) {};
    \node (8) at (7,6) {};
    \node (6) at (8,7) {};
    \node (3) at (9,8) {};
    \node (2) at (10,9) {};
\end{scope}

	\draw (0,0) rectangle (9);

	\node at ([shift={(120:.5)}]1) {\LARGE 1};
	\node at ([shift={(-60:.5)}]2) {\LARGE 2};
	\node at ([shift={(-60:.5)}]3) {\LARGE 3};
	\node at ([shift={(120:.5)}]4) {\LARGE 4};
	\node at ([shift={(120:.5)}]5) {\LARGE 5};
	\node at ([shift={(-60:.5)}]6) {\LARGE 6};
	\node at ([shift={(120:.5)}]7) {\LARGE 7};
	\node at ([shift={(-60:.5)}]8) {\LARGE 8};
	\node at ([shift={(120:.5)}]9) {\LARGE 9};
	\node at ([shift={(0:.6)}]10) {\LARGE 10};

\begin{scope}[-,
              every edge/.style={draw=orange,very thick}]
    \path [-] (6.5,0) edge node {} (6.5,10);
\end{scope}

\end{tikzpicture}\\
Now increment $\t$ to 1; since $|\SW(10)\cap\S| > f_\t$ and $|\NE(10)\cap\S| > f_\t$, player A keeps playing the $f$-Middling strategy mode.
At this point we have $(a_x,b_x) \times (a_y,b_y) = (0,x(9)) \times (0,y(9))$, $\S = \{1,4,5,7\}$, $\N = \{9,8,6,3,2\}$, and $\W = \{10\}$.

Player A keeps playing in a middlemost column of $\S$ until player B deviates again, say on turn 16:

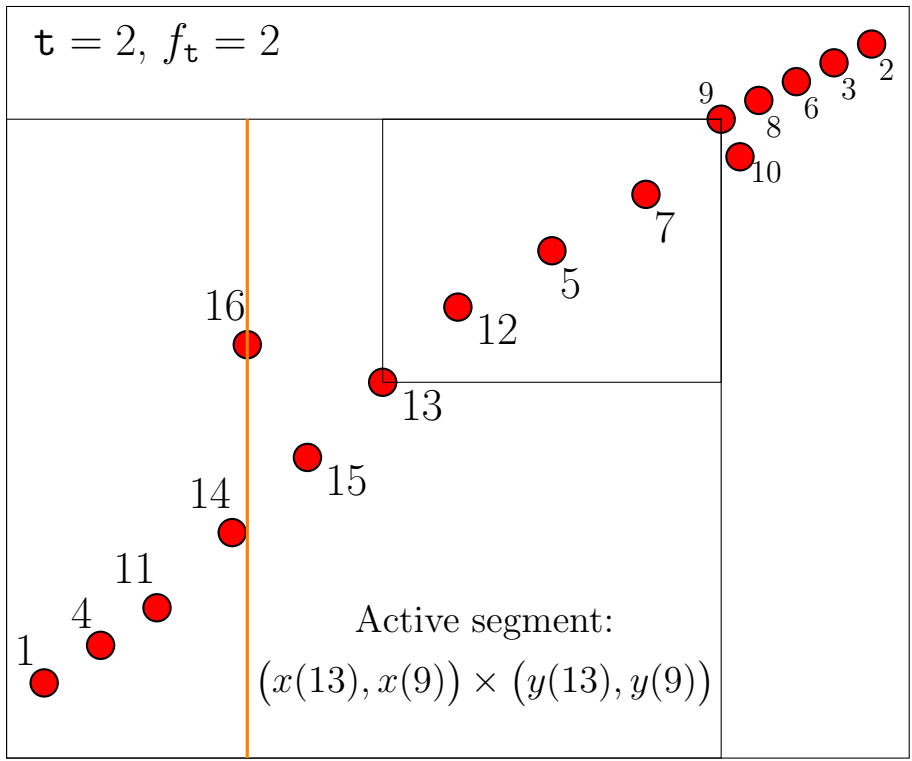
\begin{tikzpicture}

	\draw (0,0) rectangle (12,10);
	\node at (2,9.5) {\LARGE $\t = 2$, $f_\t = 2$};
	\node at (6.35,1.8) {\Large Active segment:};
	\node at (6.35,1) {\Large $\big( x(13),x(9) \big) \times \big( y(13),y(9) \big)$};
    
\begin{scope}[every node/.style={circle,thick,draw,shade,shading=axis,left color=red,right color=red,shading angle=90}]
    \node (1) at (0.5,1) {};
    \node (4) at (1.25,1.5) {};
    \node (11) at (2,2) {};
    \node (14) at (3,3) {};
    \node (16) at (3.2,5.5) {};
    \node (15) at (4,4) {};
    \node (13) at (5,5) {};
    \node (12) at (6,6) {};
    \node (5) at (7.25,6.75) {};
    \node (7) at (8.5,7.5) {};
    
    \node (9) at (9.5,8.5) {};
    \node (10) at (9.75,8) {};
    \node (8) at (10,8.75) {};
    \node (6) at (10.5,9) {};
    \node (3) at (11,9.25) {};
    \node (2) at (11.5,9.5) {};
\end{scope}

	\draw (0,0) rectangle (9);
	\draw (13) rectangle (9);

	\node at ([shift={(120:.5)}]1) {\LARGE 1};
	\node at ([shift={(-60:.4)}]2) {\large 2};
	\node at ([shift={(-60:.4)}]3) {\large 3};
	\node at ([shift={(120:.5)}]4) {\LARGE 4};
	\node at ([shift={(-60:.5)}]5) {\LARGE 5};
	\node at ([shift={(-60:.4)}]6) {\large 6};
	\node at ([shift={(-60:.5)}]7) {\LARGE 7};
	\node at ([shift={(-60:.4)}]8) {\large 8};
	\node at ([shift={(120:.4)}]9) {\large 9};
	\node at ([shift={(-30:.4)}]10) {\large 10};
	\node at ([shift={(120:.6)}]11) {\LARGE 11};
	\node at ([shift={(-30:.6)}]12) {\LARGE 12};
	\node at ([shift={(-30:.6)}]13) {\LARGE 13};
	\node at ([shift={(120:.6)}]14) {\LARGE 14};
	\node at ([shift={(-30:.6)}]15) {\LARGE 15};
	\node at ([shift={(120:.6)}]16) {\LARGE 16};

\begin{scope}[-,
              every edge/.style={draw=orange,very thick}]
    \path [-] (3.2,0) edge node {} (3.2,8.5);
\end{scope}

\end{tikzpicture}\\
Now increment $\t$ to 2; since $|\SW(16)\cap\S| > f_\t$ and $|\NE(16)\cap\S| > f_\t$, player A keeps playing the $f$-Middling strategy mode.
At this point we have $(a_x,b_x) \times (a_y,b_y) = (x(13),x(9)) \times (y(13),y(9))$, $\S = \{5,7,12\}$, $\N = \{9,8,6,3,2,1,4,11,14,15,13\}$, and $\W = \{10,16\}$.

Player A keeps playing in a middlemost column of $\S$ until player B deviates again, say on turn 19:

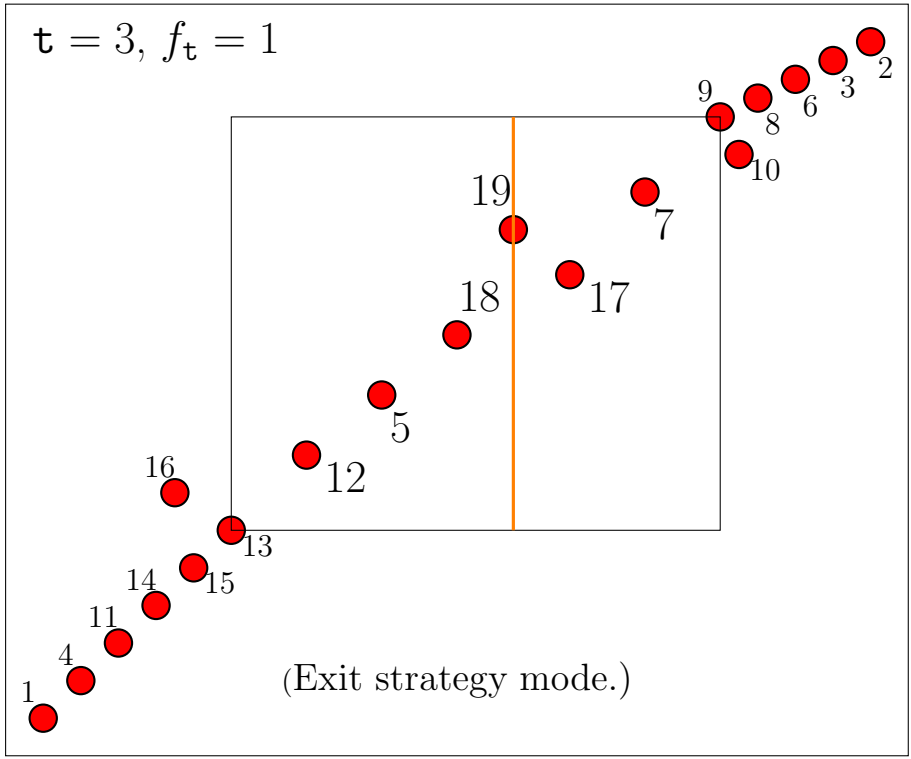
\begin{tikzpicture}

	\draw (0,0) rectangle (12,10);
	\node at (2,9.5) {\LARGE $\t = 3$, $f_\t = 1$};
	\node at (6,1) {(\Large Exit strategy mode.)};
    
\begin{scope}[every node/.style={circle,thick,draw,shade,shading=axis,left color=red,right color=red,shading angle=90}]
    \node (1) at (0.5,0.5) {};
    \node (4) at (1,1) {};
    \node (11) at (1.5,1.5) {};
    \node (14) at (2,2) {};
    \node (16) at (2.25,3.5) {};
    \node (15) at (2.5,2.5) {};
    \node (13) at (3,3) {};
    
    \node (12) at (4,4) {};
    \node (5) at (5,4.8) {};
    \node (18) at (6,5.6) {};
    \node (19) at (6.75,7) {};
    \node (17) at (7.5,6.4) {};
    \node (7) at (8.5,7.5) {};
    
    \node (9) at (9.5,8.5) {};
    \node (10) at (9.75,8) {};
    \node (8) at (10,8.75) {};
    \node (6) at (10.5,9) {};
    \node (3) at (11,9.25) {};
    \node (2) at (11.5,9.5) {};
\end{scope}

	\draw (13) rectangle (9);

	\node at ([shift={(120:.4)}]1) {\large 1};
	\node at ([shift={(-60:.4)}]2) {\large 2};
	\node at ([shift={(-60:.4)}]3) {\large 3};
	\node at ([shift={(120:.4)}]4) {\large 4};
	\node at ([shift={(-60:.5)}]5) {\LARGE 5};
	\node at ([shift={(-60:.4)}]6) {\large 6};
	\node at ([shift={(-60:.5)}]7) {\LARGE 7};
	\node at ([shift={(-60:.4)}]8) {\large 8};
	\node at ([shift={(120:.4)}]9) {\large 9};
	\node at ([shift={(-30:.4)}]10) {\large 10};
	\node at ([shift={(120:.4)}]11) {\large 11};
	\node at ([shift={(-30:.6)}]12) {\LARGE 12};
	\node at ([shift={(-30:.4)}]13) {\large 13};
	\node at ([shift={(120:.4)}]14) {\large 14};
	\node at ([shift={(-30:.4)}]15) {\large 15};
	\node at ([shift={(120:.4)}]16) {\large 16};
	\node at ([shift={(-30:.6)}]17) {\LARGE 17};
	\node at ([shift={(-300:.6)}]18) {\LARGE 18};
	\node at ([shift={(120:.6)}]19) {\LARGE 19};

\begin{scope}[-,
              every edge/.style={draw=orange,very thick}]
    \path [-] (6.75,3) edge node {} (6.75,8.5);
\end{scope}

\end{tikzpicture}\\
Now increment $\t$ to 3; since $|\NE(19)\cap\S| \le f_\t$, player A now exits the $f$-Middling strategy mode.
At this point we have $(a_y,b_y) = (x(13),x(9)) \times (y(13),y(9))$, $\S = \{5,7,12,17\}$, $\N = \{9,8,6,3,2,1,4,11,14,15,13\}$, and $\W = \{10,16\}$.

\vspace*{2cm}

\noindent \textbf{Example 2}

Now we illustrate how part of a game of $A_{m,3}$ might proceed when player A uses the $w$-Barb strategy mode, with $w = 4$.

Suppose the game is underway with the following point set.
Then player A chooses the column between $r_1$ and $\hat{r}_1$:

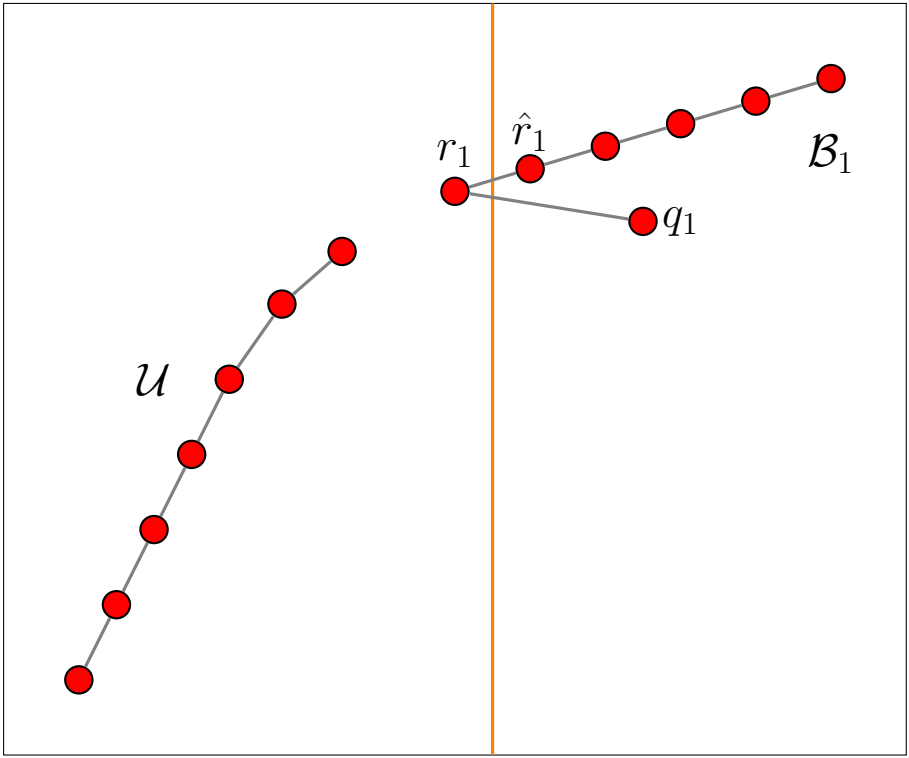
\begin{tikzpicture}

	\draw (0,0) rectangle (12,10);
    
\begin{scope}[every node/.style={circle,thick,draw,shade,shading=axis,left color=red,right color=red,shading angle=90}]
    \node (1) at (1,1) {};
    \node (2) at (1.5,2) {};
    \node (3) at (2,3) {};
    \node (4) at (2.5,4) {};
    \node (5) at (3,5) {};
    \node (6) at (3.7,6) {};
    \node (7) at (4.5,6.7) {};
    
    \node (8) at (6,7.5) {};
    \node (9) at (7,7.8) {};
    \node (10) at (8,8.1) {};
    \node (11) at (9,8.4) {};
    \node (12) at (10,8.7) {};
    \node (13) at (11,9) {};
    \node (q1) at (8.5,7.1) {};
\end{scope}

	\node at ([shift={(90:.5)}]8) {\LARGE $r_1$};
	\node at ([shift={(90:.5)}]9) {\LARGE $\hat{r}_1$};
	\node at ([shift={(0:.5)}]q1) {\LARGE $q_1$};
	
	\node at (11,8) {\LARGE $\B_1$};
	\node at (2,5) {\LARGE $\U$};

\begin{scope}[-,
              every edge/.style={draw=orange,very thick}]
    \path [-] (6.5,0) edge node {} (6.5,10);
\end{scope}

\begin{scope}[-,
              every edge/.style={draw=gray,very thick}]
    \path [-] (q1) edge node {} (8);
    \path [-] (8) edge node {} (9);
    \path [-] (9) edge node {} (10);
    \path [-] (10) edge node {} (11);
    \path [-] (11) edge node {} (12);
    \path [-] (12) edge node {} (13);
    
    \path [-] (1) edge node {} (2);
    \path [-] (2) edge node {} (3);
    \path [-] (3) edge node {} (4);
    \path [-] (4) edge node {} (5);
    \path [-] (5) edge node {} (6);
    \path [-] (6) edge node {} (7);
\end{scope}

\end{tikzpicture}\\
Then suppose player B chooses a $y$-value that results in the point $q$ which is 3 notches below $r_1$ with respect to $\U$.
Since $3 \ge w-1$, we go to Step~\ref{Bstep:stepdown} and then return to Step~\ref{Bstep:play} with the following point set.
Then player A chooses the column between $r_2$ and $\hat{r}_2$:

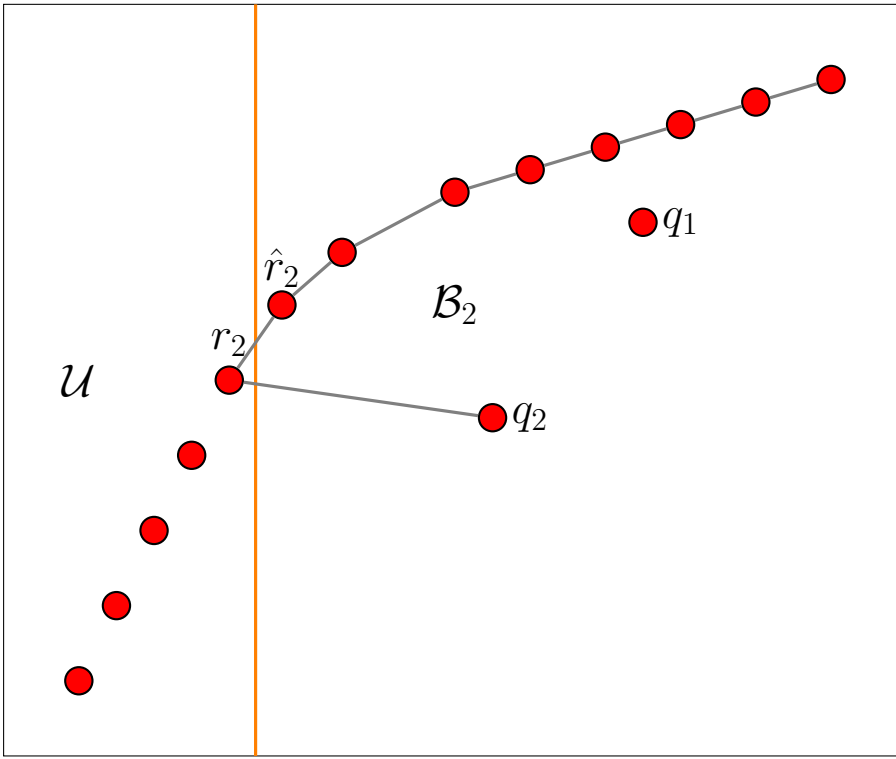
\begin{tikzpicture}

	\draw (0,0) rectangle (12,10);
    
\begin{scope}[every node/.style={circle,thick,draw,shade,shading=axis,left color=red,right color=red,shading angle=90}]
    \node (1) at (1,1) {};
    \node (2) at (1.5,2) {};
    \node (3) at (2,3) {};
    \node (4) at (2.5,4) {};
    
    \node (5) at (3,5) {};
    \node (6) at (3.7,6) {};
    \node (7) at (4.5,6.7) {};
    
    \node (8) at (6,7.5) {};
    \node (9) at (7,7.8) {};
    \node (10) at (8,8.1) {};
    \node (11) at (9,8.4) {};
    \node (12) at (10,8.7) {};
    \node (13) at (11,9) {};
    \node (q1) at (8.5,7.1) {};
    \node (q2) at (6.5,4.5) {};
\end{scope}

	\node at ([shift={(90:.5)}]5) {\LARGE $r_2$};
	\node at ([shift={(90:.5)}]6) {\LARGE $\hat{r}_2$};
	\node at ([shift={(0:.5)}]q1) {\LARGE $q_1$};
	\node at ([shift={(0:.5)}]q2) {\LARGE $q_2$};
	
	\node at (6,6) {\LARGE $\B_2$};
	\node at (1,5) {\LARGE $\U$};

\begin{scope}[-,
              every edge/.style={draw=orange,very thick}]
    \path [-] (3.35,0) edge node {} (3.35,10);
\end{scope}

\begin{scope}[-,
              every edge/.style={draw=gray,very thick}]
    \path [-] (q2) edge node {} (5);
    \path [-] (5) edge node {} (6);
    \path [-] (6) edge node {} (7);
    \path [-] (7) edge node {} (8);
    \path [-] (8) edge node {} (9);
    \path [-] (9) edge node {} (10);
    \path [-] (10) edge node {} (11);
    \path [-] (11) edge node {} (12);
    \path [-] (12) edge node {} (13);
    
\end{scope}

\end{tikzpicture}\\
Then suppose player B chooses a $y$-value that results in the point $q$ which is 1 notch below $r_2$ with respect to $\U$.
Since $1 < w-2$, we go to Step~\ref{Bstep:playbarb} with the following point set:

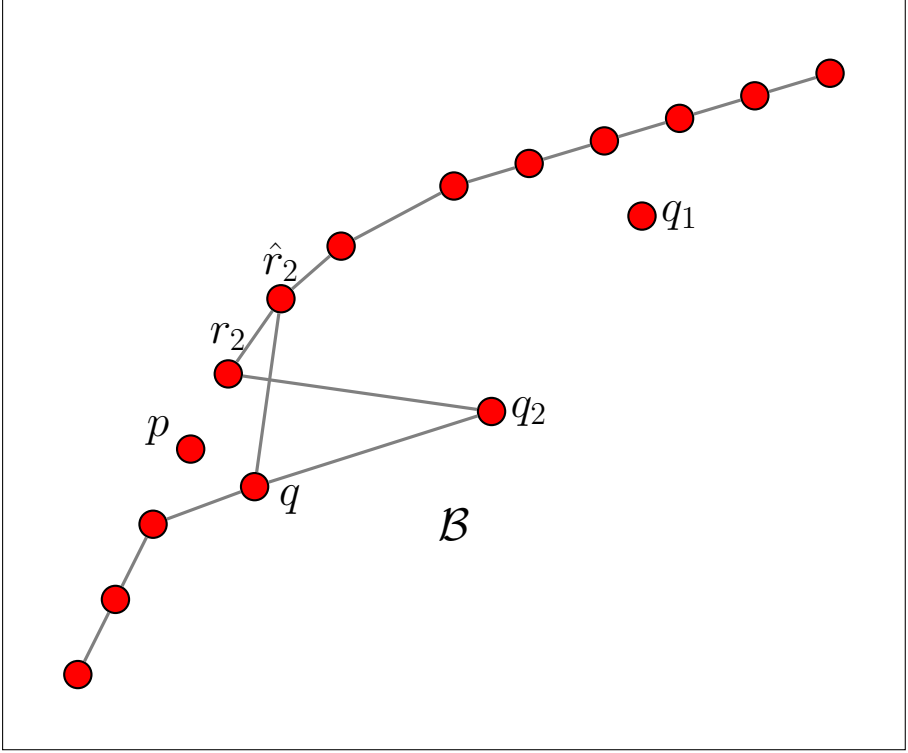
\begin{tikzpicture}

	\draw (0,0) rectangle (12,10);
    
\begin{scope}[every node/.style={circle,thick,draw,shade,shading=axis,left color=red,right color=red,shading angle=90}]
    \node (1) at (1,1) {};
    \node (2) at (1.5,2) {};
    \node (3) at (2,3) {};
    \node (4) at (2.5,4) {};
    
    \node (5) at (3,5) {};
    \node (6) at (3.7,6) {};
    \node (7) at (4.5,6.7) {};
    
    \node (8) at (6,7.5) {};
    \node (9) at (7,7.8) {};
    \node (10) at (8,8.1) {};
    \node (11) at (9,8.4) {};
    \node (12) at (10,8.7) {};
    \node (13) at (11,9) {};
    \node (q1) at (8.5,7.1) {};
    
    \node (q2) at (6.5,4.5) {};
    
    \node (q) at (3.35,3.5) {};
\end{scope}

	\node at ([shift={(90:.5)}]5) {\LARGE $r_2$};
	\node at ([shift={(90:.5)}]6) {\LARGE $\hat{r}_2$};
	\node at ([shift={(0:.5)}]q1) {\LARGE $q_1$};
	\node at ([shift={(0:.5)}]q2) {\LARGE $q_2$};
	\node at ([shift={(-20:.5)}]q) {\LARGE $q$};
	\node at ([shift={(150:.5)}]4) {\LARGE $p$};
	
	\node at (6,3) {\LARGE $\B$};

\begin{scope}[-,
              every edge/.style={draw=gray,very thick}]
    \path [-] (1) edge node {} (2);
    \path [-] (2) edge node {} (3);
    \path [-] (3) edge node {} (q); 
    \path [-] (q) edge node {} (q2);
    \path [-] (q) edge node {} (6);              
              
    \path [-] (q2) edge node {} (5);
    \path [-] (5) edge node {} (6);
    \path [-] (6) edge node {} (7);
    \path [-] (7) edge node {} (8);
    \path [-] (8) edge node {} (9);
    \path [-] (9) edge node {} (10);
    \path [-] (10) edge node {} (11);
    \path [-] (11) edge node {} (12);
    \path [-] (12) edge node {} (13);
\end{scope}

\end{tikzpicture}\\
Then player A treats $p$ and $q_1$ as a loss and plays the barb $\B$ by choosing the column between $q$ and $\hat{r}_2$.

\end{document}